\documentclass[12pt]{article}
\usepackage{amssymb}
\usepackage{latexsym}
\usepackage{amsmath}
\usepackage{amsthm}
\usepackage{bbm}
\def\e{\varepsilon}
\def\P{\mathbb P}
\def\E{\mathbb E}
\def \R{\mathbb R}

\newtheorem{lem}{Lemma}
\newtheorem{thm}{Theorem}
\newtheorem{cor}[thm]{Corollary}
\newtheorem{prop}[thm]{Proposition}

\theoremstyle{remark}

\numberwithin{equation}{section}
\numberwithin{lem}{section}
\numberwithin{thm}{section}

\begin{document}

\title{\bf On the discrepancy of random subsequences of $\{n\alpha\}$}
\author{Istv\'an Berkes\footnote{ Alfr\'ed R\'enyi Institute of Mathematics, 1053 Budapest, Re\'altanoda u.\ 13-15, Hungary. \mbox{e-mail}: \texttt{berkes.istvan@renyi.hu}. Research supported by NKFIH grant K 125569.} \, and Bence Borda\footnote{Alfr\'ed R\'enyi Institute of Mathematics, 1053 Budapest, Re\'altanoda u.\ 13-15, Hungary. \mbox{e-mail}: \texttt{borda.bence@renyi.hu}
}}
\date{}
\maketitle

\abstract{For irrational $\alpha$, $\{n\alpha\}$ is uniformly distributed mod 1 in the Weyl sense, and the asymptotic behavior of its discrepancy is completely known. In contrast, very few precise results exist for the discrepancy of subsequences $\{n_k \alpha\}$, with the exception of metric results for exponentially growing $(n_k)$.  It is therefore natural to consider random $(n_k)$, and in this paper we give nearly optimal bounds for the discrepancy of $\{n_k \alpha\}$ in the case when the gaps $n_{k+1}-n_k$ are independent, identically distributed, integer-valued random variables. As we will see, the discrepancy behavior is determined by a delicate interplay between the distribution of the gaps $n_{k+1}-n_k$ and the rational approximation properties  of $\alpha$. We also point out an interesting critical phenomenon, a sudden change of the order of magnitude of the discrepancy of $\{n_k \alpha\}$ as the Diophantine type of $\alpha$ passes through a certain critical value.}


\section{Introduction}\label{intro}

An infinite sequence $(x_k)$ of real numbers is called \textit{uniformly distributed mod} 1 if for every pair $a, b$ of
real numbers with $0\le a < b \le 1$ we have
$$\lim_{N\to\infty} \frac{1}{N}\sum_{k=1}^N I_{[a, b)} (\{x_k\})=b-a.$$
Here $\{ \cdot\}$ denotes fractional part, and $I_{[a,b)}$ is the indicator function of the interval $[a, b)$.
By Weyl's criterion \cite{WE}, a sequence $(x_k)$ is uniformly distributed mod 1 if and only if
$$\lim_{N\to\infty} \frac{1}{N} \sum_{k=1}^N e^{2\pi i hx_k}=0$$
for all integers $h \ne 0$. In particular, the sequence $\{n \alpha\}$ is uniformly distributed mod 1 for any irrational $\alpha$. It also follows that $\{n_k \alpha\}$ is uniformly distributed mod 1 for all irrational $\alpha$ for
$n_k = k^b \log^c k$ ($0 < b < 1, c\in \mathbb R$), $n_k = \log^c k$  ($c > 1$), $n_k=P(k)$, where $P$ is a nonconstant polynomial with integer coefficients. See Kuipers and Niederreiter \cite{KN} for further examples.

A natural measure of the mod 1 uniformity of an infinite sequence $(x_k)$ is the \textit{discrepancy} defined by
$$D_N(x_k) := \sup_{0\le a<b \le 1} \left| \frac{1}{N} \sum_{k=1}^N I_{[a, b)} (\{x_k\}) -(b-a)\right| \quad (N=1, 2, \ldots).$$
By Diophantine approximation theory, the order of magnitude of the discrepancy $D_N(\{n\alpha\})$ is closely connected with the rational approximation properties of $\alpha$. By a standard definition (see e.g.\ \cite{KN}), the {\it type} $\gamma$ of an irrational number $\alpha$  is the supremum of all $c$ such that $$\liminf_{q \to\infty} q^c \| q \alpha\|=0,$$
where $\left\| t \right\|$ denotes the distance from a real number $t$ to the nearest integer.
Then $\gamma\ge 1$ for all irrational $\alpha$ and by classical results (see e.g.
 \cite[Chapter 3, Theorems 3.2 and 3.3]{KN}) if $\alpha$ has finite type $\gamma$, then
\begin{equation}\label{af2}
D_N( \{n \alpha\}) =O( N^{-1/\gamma+\varepsilon}), \qquad D_N( \{n \alpha\}) =\Omega( N^{-1/\gamma-\varepsilon})
\end{equation}
for any $\varepsilon>0$.  However, the type is a rather crude measure of rational approximation and a more precise characterization can be obtained  by using a nondecreasing positive function $\psi$ such that
\begin{equation}\label{strong}
0<\liminf_{q\to\infty} \psi (q) \|q \alpha\| <\infty.
\end{equation}
Note that e.g.\ $\psi (q) = \max_{1 \le k \le q} 1/\left\| k \alpha \right\|$ satisfies \eqref{strong}, but $\psi$ is not uniquely determined by $\alpha$. For the sake of simplicity, in this paper we will focus on the case when \eqref{strong} is satisfied with $\psi (q)=q^{\gamma}$ for some $\gamma \ge 1$.
We shall say in this case that $\alpha$ {\it has strong type $\gamma$}.
As a minor change of the proof of (\ref{af2}) shows, in this case (\ref{af2}) can be improved to
$$D_N( \{n \alpha\}) =O( N^{-1/\gamma}), \qquad D_N( \{n \alpha\}) =\Omega( N^{-1/\gamma})$$
for $\gamma>1$ and
$$ D_N( \{n \alpha\}) =O\left( \frac{\log N}{N}\right)$$
for $\gamma=1$. In view of Schmidt's theorem (see e.g.\ \cite[p.\ 109]{KN}), the last bound is also optimal. Note that for any irrational $\alpha$, \eqref{strong} does not hold with any function $\psi (q) = o(q)$, and that it holds with $\psi(q)=q$ if and only if the partial quotients $a_k$ in the continued fraction of $\alpha$ remain bounded. Such irrational numbers are called \textit{badly approximable}.

In contrast to the precise results for $D_N(\{n\alpha\})$ above, much less is known about $D_N(\{n_k \alpha\})$ for general $(n_k)$. By a result of Philipp \cite{PH}, if $(n_k)$ is a sequence of positive reals with
$$n_{k+1}/n_k \ge q >1 \quad (k=1, 2, \ldots),$$
then $D_N(\{n_k\alpha\})$ satisfies the \textit{law of the iterated logarithm} (LIL):
\begin{equation}\label{ph75}
0<\limsup_{N\to\infty} \sqrt{\frac{N}{\log\log N}}  D_N (\{n_k \alpha\})<\infty
\end{equation}
for almost all $\alpha$ in the sense of the Lebesgue measure. For general $(n_k)$ growing more slowly, even sharp metric results are not available. R.\ Baker \cite{BA} proved that if $(n_k)$ is an increasing sequence of positive integers, then for any $\e >0$,
\begin{equation}\label{baker}
D_N (\{n_k \alpha \}) = O \left( N^{-1/2} (\log N)^{3/2 +\e} \right)
\end{equation}
for almost all $\alpha$, but it is not known whether the exponent $3/2$ can be improved. In the case when  $n_k$ is a polynomial with integer coefficients in $k$ of degree at least 2, Aistleitner and Larcher \cite{AL} proved the lower bound $D_N (\{n_k \alpha \}) = \Omega \left( N^{-1/2-\e} \right)$, valid  for any $\e>0$ and almost every $\alpha$. However, all these are metric results and do not give information on $D_N(\{n_k \alpha\})$ for any specific irrational $\alpha$.

Thus it is natural to consider random sequences $(n_k)$, and in this paper we consider the case when the gaps $n_{k+1}-n_k$ are independent, identically distributed (i.i.d.) random variables. That is, we are dealing with the discrepancy $D_N(\{S_k \alpha\})$, where $S_k=\sum_{j=1}^k X_j$ with i.i.d.\ random variables $X_1, X_2, \ldots$, i.e.\ $S_k$ is a random walk. In a recent paper \cite{BB} the authors proved the law of the iterated logarithm
\[ 0 < \limsup_{N \to \infty} \frac{\left| \sum_{k=1}^N e^{2 \pi i S_k \alpha} \right|}{\sqrt{N \log \log N}} < \infty \quad \text{a.s.} \]
whenever $\exp (2 \pi i X_1 \alpha)$ is nondegenerate (i.e.\ it does not equal a constant with probability 1). Note that a.s.\ (almost surely) means that the given event has probability 1 in the space of the random walk $S_k$. From Koksma's inequality \cite[Chapter 2, Corollary 5.1]{KN}, we thus obtain the following general lower estimate.

\begin{prop}\label{generallower} Let $X_1, X_2, \dots$ be i.i.d.\ random variables, let $S_k=\sum_{j=1}^k X_j$ and let $\alpha \in \mathbb{R}$. If $\exp (2 \pi i X_1 \alpha)$ is nondegenerate, then
\[ D_N(\{ S_k \alpha \}) = \Omega \left( \sqrt{\frac{\log \log N}{N}} \right) \quad \text{a.s.} \]
\end{prop}

The sharpness of Proposition \ref{generallower} follows from a result of Schatte \cite{SCH3}, who proved that if
\begin{equation}\label{SCH}
\sup_{0\le x \le 1} | \P (\{S_k\alpha\}<x)-x| = O(k^{-5/2}),
\end{equation}
then we have
\begin{equation}\label{LILdiscr}
0<\limsup_{N\to\infty} \sqrt{\frac{N}{\log\log N}} D_N(\{S_k \alpha\})<\infty \quad \text{a.s.}
\end{equation}
Condition (\ref{SCH}) is satisfied for all $\alpha \neq 0$ if the distribution of $X_1$ is absolutely continuous, in which case the convergence speed in (\ref{SCH}) is exponential. Berkes and Raseta \cite{BR} showed that in the absolutely continuous case the LIL (\ref{LILdiscr}) also holds for the $L_p$ discrepancy of $\{S_k \alpha\}$, $1\le p<\infty$, and for other functionals of the path $\{S_k \alpha\}, 1\le k \le N$.  Improving results of Schatte \cite{SCH1} and Su \cite{SU}, in \cite{BB2} we gave optimal bounds for the quantity on the left hand side of (\ref{SCH}) in the case when $X_1$ is an integer-valued random variable having a finite variance, or having heavy tails satisfying
\begin{equation}\label{heavy}
c_1 x^{-\beta} \le \P (|X_1| \ge x) \le c_2 x^{-\beta}
\end{equation}
for all $x>0$ with some constants $c_1, c_2>0$ and $0<\beta<2$. These results imply that the LIL (\ref{LILdiscr}) also holds if $\alpha$ has strong type $\gamma$ and $X_1$ is an integer-valued random variable satisfying (\ref{heavy}) with $\beta \le 2/(5\gamma)$ (see the last paragraph of Subsection \ref{HT}). In this  case $S_n$ grows, in a stochastic sense, with  the polynomial speed $n^{1/\beta}$ and this result can be considered as the stochastic analogue of Philipp's lacunary result (\ref{ph75}).
On the other hand, the results of  \cite{BB2} also show that (\ref{SCH}) cannot hold if $X_1$ has a finite variance, in which case $S_n$ grows at most linearly.  In this case the problem of asymptotic behavior of $D_N(\{S_k \alpha\})$ becomes considerably harder and will be studied in the present paper.

Upper bounds for $D_N(\{S_k \alpha\})$ for general random walks in terms of the growth rate of the sums
$$ \sum_{h=1}^H \frac{1}{h|1-\varphi(2 \pi h\alpha)|} \quad \text{and} \quad \sum_{h=1}^H \frac{1}{h|1-\varphi(2 \pi h\alpha)|^{1/2}} $$
were given in Weber \cite{W} and Berkes and Weber \cite{BW}. Here $\varphi$ denotes the characteristic function of $X_1$. In particular, in \cite{W} it is shown that if $X_1$ is integer-valued, $S_k/k^{1/\beta}$ converges in distribution to a stable law with parameter $0<\beta<1$ and $\alpha$ satisfies $\left\| q \alpha \right\| \ge C q^{-\gamma}$ for every $q \in \mathbb{N}$ with some $\gamma >1$ and $C>0$, then
\begin{equation}\label{Weberbound}
D_N( \{S_k\alpha\})=O \left( N^{- 1/(1+\gamma)} \log^{2+\varepsilon} N\right) \quad \text{a.s.}
\end{equation}
for any $\varepsilon>0$. The same upper bound holds if instead of the distributional convergence of $S_k/k^{1/\beta}$ we assume $\mathbb{E}X_1 \neq 0$ and $\mathbb{E}|X_1|<\infty$. For nearly optimal improvements of this estimate, see Propositions \ref{mainprop1} and \ref{mainprop2} below.

The main focus of this paper is to study the discrepancy of $\{ S_k \alpha \}$ in the case when $X_1$ is an integer-valued random variable, and $\alpha$ is irrational. The most interesting case is $X_1>0$, when $\{ S_k \alpha \}$ is in fact a random subsequence of $\{n \alpha \}$, but in general we will allow $X_1$ to take negative integers as well. Before we formulate our general results, we discuss here the simple special case when $X_1$ takes the values 1 and 2 with probability $1/2$ each. The corresponding sequence $\{ S_k \alpha \}$ is arguably the simplest random subsequence of $\{ n \alpha \}$.

\begin{prop}\label{mainprop1} Let $X_1, X_2, \dots$ be i.i.d.\ random variables such that $\P (X_1=1)=\P (X_1=2)=1/2$, let $S_k=\sum_{j=1}^k X_j$, and let $\alpha \in \R$ be irrational.
\begin{enumerate}
\item[(i)] If $\left\| q \alpha \right\| \ge C q^{-2}$ for every $q \in \mathbb{N}$ with some constant $C>0$, then $D_N=D_N(\{ S_k \alpha \})$ satisfies
\[ D_N = O \left( \sqrt{\frac{\log \log N}{N}} \log N \right), \quad D_N = \Omega \left( \sqrt{\frac{\log \log N}{N}} \right) \quad \text{a.s.} \]

\item[(ii)] If $0< \liminf_{q \to \infty} q^{\gamma} \left\| q \alpha \right\| < \infty$ with some $\gamma >2$, then $D_N=D_N(\{ S_k \alpha \})$ satisfies
\[ D_N = O \left( \left( \frac{\log \log N}{N} \right)^{1/\gamma} \right), \quad D_N = \Omega \left( \frac{1}{N^{1/\gamma}} \right) \quad \text{a.s.} \]
\end{enumerate}
\end{prop}

For an irrational $\alpha$ with strong type $\gamma$, the estimates in (i) hold if $1 \le \gamma \le 2$, while those in (ii) hold if $\gamma >2$. Thus the behavior of $D_N(\{S_k \alpha \})$ changes at the critical value $\gamma=2$. It would not be difficult to generalize (ii) to irrational $\alpha$ satisfying \eqref{strong} with an arbitrary $\psi (q)$ increasing faster than $q^2$. In this case the estimates for $D_N(\{ S_k \alpha \})$ would be given in terms of the inverse function $\psi^{-1}$.

The estimates in (i) apply to every algebraic irrational $\alpha$, as well as to almost every $\alpha$ in the sense of the Lebesgue measure. Indeed, a celebrated theorem of Roth \cite{RO} states that any algebraic irrational $\alpha$ satisfies $\left\| q \alpha \right\| \ge C q^{-(1+\e)}$ with some constant $C=C(\alpha, \e)>0$, where $\e>0$ is arbitrary. Furthermore, according to the Jarn\'ik--Besicovitch theorem \cite{BE}, the set of all $\alpha \in \R$ for which $\liminf_{q \to \infty} q^{\gamma} \left\| q \alpha \right\| <\infty$ has Hausdorff dimension $2/(\gamma+1)$. Thus except for a set of Hausdorff dimension 2/3 (and hence Lebesgue measure $0$), every $\alpha \in \R$ satisfies the Diophantine condition in (i).

Note that the exponent 1 of the log in the upper estimate in (i) is smaller than the exponent 3/2 in Baker's estimate (\ref{baker}), and thus random sequences give a better discrepancy bound.

\section{Results}\label{mainresults}

\subsection{Heavy-tailed distributions}\label{HT}

Suppose that the random variable $X_1$ has a \textit{heavy-tailed} distribution, i.e.\ $\mathbb{E} X_1^2 = \infty$. For the sake of simplicity, we only formulate a result for random variables whose tail distribution decays at the rate of a power function. The indicator function of the event $E$ will be denoted by $I_E$.

\begin{prop}\label{mainprop2} Let $X_1, X_2, \dots$ be integer-valued i.i.d.\ random variables such that $c_1 x^{-\beta} \le \P (|X_1| \ge x) \le c_2 x^{-\beta}$ for all $x>0$ with some constants $0<\beta<2$ and $c_1, c_2>0$. For $1 < \beta <2$ suppose also that $\mathbb{E}X_1=0$, and for $\beta =1$ that $|\E (X_1 I_{\{ |X_1| < x \}})| \le c_3$ for all $x>0$ with some constant $c_3>0$. Let $S_k=\sum_{j=1}^k X_j$, and let $\alpha \in \mathbb{R}$ be irrational.
\begin{enumerate}
\item[(i)] If $\left\| q \alpha \right\| \ge C q^{-2/\beta}$ for every $q \in \mathbb{N}$ with some constant $C>0$, then $D_N=D_N(\{ S_k \alpha \})$ satisfies
\[ D_N = O \left( \sqrt{\frac{\log \log N}{N}} \log N \right), \quad D_N = \Omega \left( \sqrt{\frac{\log \log N}{N}} \right) \quad \text{a.s.} \]

\item[(ii)] If $0< \liminf_{q \to \infty} q^{\gamma} \left\| q \alpha \right\| < \infty$ with some $\gamma >2/\beta$, then $D_N=D_N(\{ S_k \alpha \})$ satisfies
\[ D_N = O \left( \left( \frac{\log \log N}{N} \right)^{1/(\beta \gamma)} \right), \quad D_N = \Omega \left( \frac{1}{N^{1/(\beta \gamma)}} \right) \quad \text{a.s.} \]
\end{enumerate}
\end{prop}

Here we have a dichotomy similar to that in Proposition \ref{mainprop1}, the critical value of $\gamma$ being $2/\beta$. Again, it would not be difficult to generalize (ii) to irrational $\alpha$ satisfying \eqref{strong} with an arbitrary $\psi (q)$ increasing faster than $q^{2/\beta}$. Similarly, we could derive estimates for random variables with tail distribution $c_1 \phi (x) \le \P (|X_1| \ge x) \le c_2 \phi (x)$, where $\phi (x)$ is not necessarily a power function. In this more general situation the critical order of magnitude of $\psi (q)$, where the behavior of $D_N$ changes, would not necessarily be a power function.

Note that the estimates in (i) apply to every algebraic irrational $\alpha$, as well as to almost every $\alpha$ in the sense of the Lebesgue measure.

Proposition \ref{mainprop2} applies e.g.\ to the positive integer-valued random variable $X_1$ with $\P (X_1=n)=c_{\beta}/n^{1+\beta}$, $n=1,2,\dots$, where $0< \beta < 1$. This way we obtain a random subsequence $S_k \alpha$ of $n \alpha$ increasing roughly at the polynomial speed $k^{1/\beta}$. More precisely, $S_k = O \left( k^{1/\beta + \e} \right)$ a.s.\ for any $\e>0$ but not for $\e=0$ (see e.g.\ \cite[Theorem 6.9]{P}).

In conclusion we note that Schatte's LIL under (\ref{SCH}) and Proposition 2.1 of our previous paper \cite{BB2} imply
that if in statement (i) of Proposition \ref{mainprop2} we replace the assumption $\| q \alpha \| \ge C q^{-2/\beta}$ by $ \| q \alpha\| \ge Cq^{-2/(5\beta) }$, then in the conclusion
$$ D_N = O \left( \sqrt{\frac{\log \log N}{N}} \log N \right) \quad \text{a.s.} $$
the factor $\log N$ can be dropped, resulting in a sharp LIL bound. Whether this is true under the original assumption remains open.

\subsection{The case $\E X_1^2<\infty$, $\E X_1 = 0$}

The previous result deals with the case $\E X_1^2 =\infty$, and covers the typical case when the tails of $X_1$ decrease with speed $x^{-\beta}$, $0<\beta<2$. Next, we assume $\E X_1^2<\infty$. As we will see, the results are substantially different according as $\E X_1$ equals 0 or not, and we start with the easier case $\E X_1=0$.

\begin{prop}\label{mainprop3} Let $X_1, X_2, \dots$ be nondegenerate integer-valued i.i.d.\ random variables such that $\E X_1=0$ and $\E X_1^2 < \infty$, let $S_k=\sum_{j=1}^k X_j$, and let $\alpha \in \R$ be irrational.
\begin{enumerate}
\item[(i)] If $\left\| q \alpha \right\| \ge C q^{-1}$ for every $q \in \mathbb{N}$ with some constant $C>0$, then $D_N=D_N(\{ S_k \alpha \})$ satisfies
\[ D_N = O \left( \sqrt{\frac{\log \log N}{N}} \log^2 N \right), \quad D_N = \Omega \left( \sqrt{\frac{\log \log N}{N}} \right) \quad \text{a.s.} \]

\item[(ii)] If $0< \liminf_{q \to \infty} q^{\gamma} \left\| q \alpha \right\| < \infty$ with some $\gamma >1$, then $D_N=D_N(\{ S_k \alpha \})$ satisfies
\[ D_N = O \left( \left( \frac{\log \log N}{N} \right)^{1/(2 \gamma)} \right), \quad D_N = \Omega \left( \frac{1}{N^{1/(2\gamma)}} \right) \quad \text{a.s.} \]
\end{enumerate}
\end{prop}

The dichotomy is less pronounced here than in the previous propositions. Formally, the critical value is now $\gamma =1$. Thus (i) only applies to badly approximable irrationals, but not to almost every $\alpha$.

Note that the factor $\log ^2 N$ in the upper estimate in (i) is greater than the factor $(\log N)^{3/2+\varepsilon}$ in Baker's bound (\ref{baker}). However, Baker's bound does not apply to $\{S_k \alpha \}$, since $\mathbb{E}X_1=0$ implies that $S_k$ cannot be an increasing sequence. Additionally, the set of all badly approximable $\alpha$ is of measure 0, and Baker's estimate provides no information on what happens in such sets. As more than one result in our paper shows, discrepancy estimates in sets of zero measure can be much worse than the ``typical'' behavior.

\subsection{The case $\E X_1^2<\infty$, $\E X_1 \ne 0$}

The relation $\E X_1\neq 0$ holds in particular if $X_1>0$, when the sequence $S_k$ is increasing with probability 1, a natural situation since in this case $\{S_k \alpha\}$ is a random subsequence of $\{n \alpha\}$. As we will see, this case is considerably more involved, and we can prove almost tight estimates for the discrepancy only for certain special distributions, such as in Proposition \ref{mainprop1}.

In Section \ref{Examples} we will see further examples for which Proposition \ref{mainprop1} holds. For example, this is the case if $\P (X_1=a)=\P (X_1=b)=1/2$ for some $a,b \in \mathbb{Z}$, $a \not\equiv b \pmod{2}$, and also if $\E |X_1|<2 \P (X_1=1)$. However, we do not have a complete characterization of distributions for which the estimates in Proposition \ref{mainprop1} are valid. In the (admittedly most interesting) case  $\E X_1^2 < \infty$, $\E X_1 \neq 0$, for an irrational $\alpha$ of strong type $\gamma >1$ in general we only know that $D_N (\{ S_k \alpha \})$ is, up to logarithmic factors, at most $N^{-1/(\gamma+1)}$ because of \eqref{Weberbound}, and at least $N^{-\tau}$ with $\tau=\min \{1/2, 1/\gamma \}$ because of Proposition \ref{generallower} and Lemma \ref{proplower2} below. Thus there is a gap between the exponents of $N$ in the upper and lower estimates, and the precise exponent remains open.

\subsection{Main theorem}

As we have seen, the order of magnitude of the discrepancy $D_N(\{S_k \alpha\})$ is sensitive to the distribution of $X_1$ and the Diophantine properties of $\alpha$. Theorem \ref{theoremA} below, which is the main result of our paper, provides criteria in terms of the characteristic function $\varphi$ of $X_1$. As we will see, these criteria cover all the above-mentioned classes and actually more.

\begin{thm}\label{theoremA} Let $X_1, X_2, \ldots$ be i.i.d.\ random variables with characteristic function $\varphi$, and let $S_k=\sum_{j=1}^k X_j$. Let $\alpha \in \R$ be irrational such that $\left\| q \alpha \right\| \ge C q^{-\gamma}$ for every $q \in \mathbb{N}$ with some constants $\gamma \ge 1$ and $C>0$.

\begin{enumerate}
\item[(i)] Suppose there exist real numbers $0<\beta \le 2$, $c>0$ and an integer $d>0$ such that for any $x\in \R$,
\begin{equation}\label{firstcond}
1-|\varphi(2\pi x)| \ge c \| dx\|^\beta .
\end{equation}
Then, with $s=1$ if \, $0<\beta<2$, and $s=2$ if \, $\beta=2$,
\begin{equation}\label{main}
D_N (\{S_k\alpha\})= \left\{ \begin{array}{ll} O \left( \sqrt{\frac{\log \log N}{N}} \log^s N \right) \quad \text{a.s.} & \text{if } 1\le \gamma\le 2/\beta ,\\ O \left( \left( \frac{\log \log N}{N} \right)^{1/(\beta \gamma)} \right) \quad \text{a.s.} & \text{if } \gamma >2/\beta . \end{array} \right.
\end{equation}

\item[(ii)] Suppose there exist a real number $c>0$ and an integer $d>0$ such that for any $x,y \in \R$,
\begin{equation}\label{secondcond}
|\varphi(2\pi x) - \varphi(2\pi y)| \ge c \| d(x-y)\|.
\end{equation}
Then
\begin{equation*}
D_N (\{S_k\alpha\})= \left\{ \begin{array}{ll} O \left( \sqrt{\frac{\log \log N}{N}} \log N \right) \quad \text{a.s.} & \text{if } 1 \le \gamma \le 2,\\ O \left( \left( \frac{\log \log N}{N} \right)^{1/\gamma} \right) \quad \text{a.s.} & \text{if } \gamma >2. \end{array} \right.
\end{equation*}
\end{enumerate}
\end{thm}

\noindent Conditions \eqref{firstcond} and \eqref{secondcond} are not standard in probability theory, therefore we offer some insight into their behavior in Section \ref{Examples}. As we will see in Proposition \ref{examplesprop} (i), Theorem \ref{theoremA} (i) with $\beta=2$ applies to any nondegenerate integer-valued $X_1$, making it our most general upper estimate.

Although we did not assume in Theorem \ref{theoremA} that $X_1$ is integer-valued, and indeed there exist non-integer-valued distributions satisfying \eqref{firstcond} or \eqref{secondcond}, the estimates, while valid, might be far from optimal in the non-integral case. Note that the upper bounds in Proposition \ref{mainprop1} will follow from Theorem \ref{theoremA} (ii); the upper bounds in Proposition \ref{mainprop2} will be a corollary of Theorem \ref{theoremA} (i) with $0<\beta<2$; finally, the upper bounds in Proposition \ref{mainprop3} will be deduced from Theorem \ref{theoremA} (i) with $\beta=2$. The lower bounds in Propositions \ref{mainprop1}, \ref{mainprop2} and \ref{mainprop3} are either a special case of Proposition \ref{generallower}, or follow from a simple argument based on the growth rate of $S_k$ (see Lemmas \ref{proplower2} and \ref{proplower3} below).

Our proof of Theorem \ref{theoremA} is based on the Erd\H{o}s--Tur\'an inequality, which states that for any sequence $(x_k)$ of reals and any $H \in \mathbb{N}$
\begin{equation}\label{ErdosTuran}
D_N(x_k) \le C \left( \frac{1}{H} + \sum_{h=1}^H \frac{1}{h} \left| \frac{1}{N} \sum_{k=1}^N e^{2 \pi i h x_k} \right| \right)
\end{equation}
with a universal constant $C>0$. The free parameter $H$ can be chosen arbitrarily to optimize the estimate. Note that the same exponential sum shows up in Weyl's criterion. To estimate $D_N (\{ S_k \alpha \})$, we therefore need to study
\begin{equation}\label{expsum}
\sum_{k=1}^N e^{2 \pi i S_k h \alpha} ,
\end{equation}
and this is why it was natural to state the conditions of Theorem \ref{theoremA} in terms of the characteristic function $\varphi$ of $X_1$. The same approach was followed in Weber \cite{W} and Berkes and Weber \cite{BW}, which were the starting point for our investigations. The various arithmetic and metric upper bounds for $D_N(\{S_k \alpha\})$ in \cite{W} and \cite{BW} were based on estimates for the second and fourth moments of \eqref{expsum}. The improvements in the present paper depend on sharp asymptotic estimates for the $2p$th moments of \eqref{expsum} for $p =O(\log\log N)$, a technique going back to Erd\H{o}s and G\'al \cite{EG} and which, as we will see, presents considerable combinatorial difficulties. A crucial ingredient of the argument will be a sharp estimate for Diophantine sums
\[ \sum_{h=1}^H \frac{1}{h \|h \alpha\|^b} \quad (0 < b \le 1) \]
(see Proposition \ref{generaldioph} and Corollary \ref{gammadioph}), which is of independent interest.

\section{The moments of an exponential sum}

Let $X_1, X_2, \dots$ be i.i.d.\ random variables, $S_k = \sum_{j=1}^k X_j$ and $\alpha \in \R$. In this section we estimate the moments
\begin{equation}\label{2pmoment}
\E \left| \sum_{k=m+1}^{m+n} e^{2 \pi i S_k \alpha} \right|^{2p}
\end{equation}
where $p \ge 1$ is an integer. The order of magnitude of \eqref{2pmoment} depends on a delicate interplay between the distribution of the random variable $X_1$ and the value of $\alpha$. Our main focus is on the case when $X_1$ is integer-valued, and $\alpha$ is irrational.

To get a basic understanding of \eqref{2pmoment}, consider the simplest case $p=1$. Expanding the square we get
\[ \E \left| \sum_{k=m+1}^{m+n} e^{2 \pi i S_k \alpha} \right|^2 = \sum_{k_1, k_2=m+1}^{m+n} \E e^{2 \pi i (S_{k_1} - S_{k_2}) \alpha} . \]
We need to decompose this sum into three parts, according to the cases $k_1=k_2$, $k_1<k_2$ and $k_1>k_2$. The terms with $k_1=k_2$ are simply 1. In the other two cases, using the independence of $X_1, X_2, \dots$ we have
\begin{equation}\label{k1k2cases}
\E e^{2 \pi i (S_{k_1} - S_{k_2}) \alpha} = \left\{ \begin{array}{ll} \varphi (-2 \pi \alpha)^{k_2-k_1} & \textrm{if } k_1<k_2, \\ \varphi (2 \pi \alpha)^{k_1-k_2} & \textrm{if } k_1>k_2. \end{array} \right.
\end{equation}
It is now easy to sum over all pairs $m+1 \le k_1,k_2 \le m+n$ and obtain an explicit formula for \eqref{2pmoment} in the case $p=1$.

The basic tool for the case $p>1$ is a generalization of the decomposition above which enables an evaluation similar to \eqref{k1k2cases} of the terms in the expanded sum. The number of cases will obviously be much larger than 3, in fact it will be almost as large as $(2p)^{2p}$.

We are ultimately interested in the discrepancy of the sequence $\{ S_k \alpha \}$. To use \eqref{ErdosTuran} with $x_k=S_k \alpha$ for a specific $\alpha$, we therefore need to estimate \eqref{2pmoment} not only for $\alpha$, but for every integral multiple of $\alpha$ as well.
The main difficulty of this section is thus that our estimate of \eqref{2pmoment} cannot contain any implied constant depending on $\alpha$, it has to be completely explicit.

\subsection{Two estimates of the moments}

We now prove two estimates of \eqref{2pmoment} under two different conditions on the distribution of $X_1$. In the proofs we will often use the fact that $\left\| \cdot \right\|$ is symmetric and subadditive, i.e.\ $\left\| -x \right\| = \left\| x \right\|$ and $\left\| x+y \right\| \le \left\| x \right\| + \left\| y \right\|$ for any $x,y \in \R$, and that the characteristic function $\varphi$ of any probability distribution satisfies $\varphi (-x)=\bar{\varphi} (x)$ and $|\varphi (x)| \le 1$ for any $x \in \R$.

\begin{prop}\label{momentestimate} Let $X_1, X_2, \dots$ be i.i.d.\ random variables with characteristic function $\varphi$, and let $S_k=\sum_{j=1}^k X_j$.
\begin{enumerate}
\item[(i)] Suppose that there exist real constants $0<\beta \le 2$ and $c,d>0$ such that \eqref{firstcond} holds for any $x \in \R$. For any $\alpha \in \R$ such that $d \alpha \not\in \mathbb{Z}$, and any integers $m \ge 0$ and $n,p \ge 1$,
\begin{equation}\label{mom(i)}
 \E \left| \sum_{k=m+1}^{m+n} e^{2 \pi i S_k \alpha} \right|^{2p} \le (8p)^{2p} \max_{1 \le r \le p} \frac{n^r}{r! \left( c \left\| d \alpha \right\|^{\beta} \right)^{2p-r}} .
\end{equation}
\item[(ii)] Suppose that there exist real constants $c,d>0$ such that \eqref{secondcond} holds for any $x,y \in \R$. For any $\alpha \in \R$ such that $d \alpha \not\in \mathbb{Z}$, and any integers $m \ge 0$ and $n,p \ge 1$,
\begin{equation}\label{mom(ii)}
 \E \left| \sum_{k=m+1}^{m+n} e^{2 \pi i S_k \alpha} \right|^{2p} \le (4p)^{2p} \sum_{r=0}^p \frac{n^r}{r! \left( c \left\| d \alpha \right\| \right)^{2p-r}} .
\end{equation}
\end{enumerate}
\end{prop}

\begin{proof} Let us expand the power to obtain
\begin{equation}\label{2pexpand}
\E \left| \sum_{k=m+1}^{m+n} e^{2 \pi i S_k \alpha} \right|^{2p} = \sum_{k_1, \dots , k_{2p}=m+1}^{m+n} \E e^{2 \pi i (S_{k_1} - S_{k_2} + \cdots + S_{k_{2p-1}} - S_{k_{2p}}) \alpha} .
\end{equation}
In order to compute the expected value, we need to write the exponent as a sum of independent random variables. To this end, let us say that $P=(P_1, \dots, P_s)$ is an \textit{ordered partition} of the set $[2p]$, where $[N]$ denotes the set $\left\{ 1, \dots , N \right\}$ for any $N \in \mathbb{N}$, if $P_1, \dots , P_s$ are pairwise disjoint, nonempty subsets of $[2p]$ such that $\bigcup_{j=1}^s P_j = [2p]$. We can associate an ordered partition to every $2p$-tuple $k=(k_1, \dots, k_{2p})$ in a natural way: if
\begin{equation}\label{k=l}
\left\{ k_1, \dots, k_{2p} \right\} = \left\{ \ell_1, \dots, \ell_s \right\}
\end{equation}
with $\ell_1<\cdots < \ell_s$, then for any $1 \le j \le s$ let
\[ P_j (k) = \left\{ i \in [2p] \,\, : \,\, k_i= \ell_j \right\} . \]
Then $P(k)=\left( P_1(k), \dots, P_s(k) \right)$ is an ordered partition of $[2p]$. In other words, the numbers $k_1, \dots, k_{2p}$ are written in increasing order as $\ell_1 < \dots < \ell_s$ (note $s \le 2p$ where we do not necessarily have equality since $k_1, \dots, k_{2p}$ need not be distinct), and we let $P_1(k)$ be the set of indices $i$ such that $k_i$ is the smallest, we let $P_2(k)$ be the set of indices $i$ such that $k_i$ is the second smallest etc. We will decompose the sum in \eqref{2pexpand} according to the value of $P(k)$. For any given ordered partition $P$ of $[2p]$ let
\[ S(P) = \sum_{\substack{k_1, \dots , k_{2p}=m+1 \\ P(k)=P}}^{m+n} \E e^{2 \pi i (S_{k_1} - S_{k_2} + \cdots + S_{k_{2p-1}} - S_{k_{2p}}) \alpha} . \]

Let us now fix an ordered partition $P=(P_1, \dots , P_s)$ of $[2p]$. Let $k$ be such that $P(k)=P$, and let $\ell_1 < \cdots < \ell_s$ be as in \eqref{k=l}. We have
\[ S_{k_1} - S_{k_2} + \cdots + S_{k_{2p-1}} - S_{k_{2p}} = \e_1 S_{\ell_1} + \cdots + \e_s S_{\ell_s} \]
where $\e_j=\sum_{i \in P_j} (-1)^{i+1}$ for any $1 \le j \le s$. Since $\ell_1 < \cdots < \ell_s$, it is now easy to write this as a sum of independent random variables:
\[ \e_1 S_{\ell_1} + \cdots + \e_s S_{\ell_s} = c_1 \sum_{t=1}^{\ell_1} X_t + c_2 \sum_{t=\ell_1+1}^{\ell_2} X_t + \cdots + c_s \sum_{t=\ell_{s-1}+1}^{\ell_s} X_t \]
where $c_j = \e_j + \e_{j+1} + \cdots + \e_s$. Note that $\e_1, \dots, \e_s$ and $c_1, \dots, c_s$ depend only on the fixed ordered partition $P$. Therefore
\[ \E e^{2 \pi i (S_{k_1} - S_{k_2} + \cdots + S_{k_{2p-1}} - S_{k_{2p}}) \alpha} = \varphi (2 \pi c_1 \alpha)^{\ell_1} \varphi (2 \pi c_2 \alpha)^{\ell_2-\ell_1} \cdots \varphi (2 \pi c_s \alpha)^{\ell_s-\ell_{s-1}} , \]
and
\begin{equation}\label{phiexpansion}
S(P) = \sum_{m+1 \le \ell_1 < \cdots < \ell_s \le m+n} \varphi (2 \pi c_1 \alpha)^{\ell_1} \varphi (2 \pi c_2 \alpha)^{\ell_2-\ell_1} \cdots \varphi (2 \pi c_s \alpha)^{\ell_s-\ell_{s-1}} .
\end{equation}
This is the generalization of \eqref{k1k2cases} to arbitrary $p \ge 1$. We are going to estimate \eqref{phiexpansion} in two different ways, according to the hypotheses \eqref{firstcond} and \eqref{secondcond}.

First, we prove Proposition \ref{momentestimate} (i), i.e.\ we assume \eqref{firstcond}. Observe that the set
\[ B=\left\{ k \in \mathbb{Z} \,\, : \,\, \left\| dk \alpha \right\| < \frac{1}{2} \left\| d \alpha \right\| \right\} \]
contains no two consecutive integers. Indeed, if $k,k+1 \in B$, then using the symmetry and the subadditivity of $\left\| \cdot \right\|$ we would have
\[ \left\| d \alpha \right\| \le \left\| d(k+1) \alpha \right\| + \left\| -dk \alpha \right\| < \frac{1}{2} \left\| d \alpha \right\| + \frac{1}{2} \left\| d \alpha \right\| , \]
a contradiction. Clearly $0 \in B$ and $\pm 1 \not\in B$. Consider the set
\[ \left\{ 1 \le j \le s \,\, : \,\, c_j \in B \right\} = \left\{ j_1, \dots , j_r \right\} \]
where $j_1<\cdots <j_r$. Note that
\[ c_1 = \e_1 + \cdots + \e_s = \sum_{i=1}^{2p} (-1)^{i+1} =0 \in B , \]
hence $j_1=1$. Since $B$ contains no consecutive integers, for any $1 \le a \le r-1$ we have
\[ \pm 1 \neq c_{j_a} - c_{j_{a+1}} = \sum_{j_a \le j < j_{a+1}} \e_j = \sum_{i \in \bigcup_{j_a \le j < j_{a+1}} P_j} (-1)^{i+1} . \]
Similarly, $\pm 1 \not\in B$ implies
\[ \pm 1 \neq c_{j_r} = \sum_{j_r \le j \le s} \e_j = \sum_{i \in \bigcup_{j_r \le j \le s} P_j} (-1)^{i+1} . \]
Therefore $\left| \bigcup_{j_a \le j < j_{a+1}} P_j \right| \ge 2$ and $\left| \bigcup_{j_r \le j \le s} P_j \right| \ge 2$. Using the fact that $P_1,\dots, P_s$ is a partition of $[2p]$ we thus obtain
\[ 2r \le \sum_{a=1}^{r-1} \left| \bigcup_{j_a \le j < j_{a+1}} P_j \right| + \left| \bigcup_{j_r \le j \le s} P_j \right| \le 2p . \]
In other words, $c_j \in B$ for at most $p$ indices $1 \le j \le s$.

Let us now apply the triangle inequality to \eqref{phiexpansion}. For any $j \neq j_1, \dots , j_r$ we have $c_j \not\in B$, hence condition \eqref{firstcond} implies
\[ |\varphi (2 \pi c_j \alpha)| \le 1 - c \left\| d c_j \alpha \right\|^{\beta} \le 1-\frac{c}{2^{\beta}} \left\| d \alpha \right\|^{\beta} . \]
For $j=j_1, \dots , j_r$ let us use the trivial estimate $|\varphi (2 \pi c_j \alpha)| \le 1$. Recall that $j_1=1$, which means that we in fact use the trivial estimate on the first factor $\varphi (2 \pi c_1 \alpha)^{\ell_1}$. This way we obtain
\begin{equation}\label{Destimate}
\left| S(P) \right| \le \sum_{m+1 \le \ell_1 < \cdots < \ell_s \le m+n} \left( 1-\frac{c}{2^{\beta}} \left\| d \alpha \right\|^{\beta} \right)^{\sum_{j \neq j_1, \dots , j_r} \left( \ell_j - \ell_{j-1} \right)} .
\end{equation}
We need to estimate the number of indices $m+1 \le \ell_1 < \dots < \ell_s \le m+n$ for which the total exponent is some fixed integer
\begin{equation}\label{totalexp}
\ell = \sum_{\substack{1 \le j \le s \\ j \neq j_1, \dots , j_r}} \left( \ell_j - \ell_{j-1} \right) .
\end{equation}
The special indices $\ell_{j_1}, \dots, \ell_{j_r}$ can be chosen in $\binom{n}{r} \le n^r/r!$ ways. Given $\ell_{j_1}, \dots, \ell_{j_r}$, the positive integers $\ell_j-\ell_{j-1}$, $j \neq j_1, \dots, j_r$ determine all of $\ell_1, \dots, \ell_s$. The number of ways to write $\ell$ as a sum of $s-r$ nonnegative integers (where the order of the terms matter) is $\binom{\ell+s-r-1}{s-r-1}$, provided $r<s$. The number of indices $m+1 \le \ell_1 < \dots < \ell_s \le m+n$ for which \eqref{totalexp} holds is thus at most $\frac{n^r}{r!} \binom{\ell+s-r-1}{s-r-1}$, and so \eqref{Destimate} gives
\[ |S(P)| \le \sum_{\ell=0}^{\infty} \frac{n^r}{r!} \binom{\ell +s-r-1}{s-r-1} \left( 1-\frac{c}{2^{\beta}} \left\| d \alpha \right\|^{\beta} \right)^{\ell} . \]
This is in fact a well-known power series which can be obtained by differentiating the geometric series $s-r-1$ times. Hence
\[ |S(P)| \le \frac{n^r}{r! \left( \frac{c}{2^{\beta}} \left\| d \alpha \right\|^{\beta} \right)^{s-r}} \]
if $r<s$, but clearly the same is true if $r=s$ (in which case our method simply estimates the absolute value of each term of \eqref{phiexpansion} by 1). Here $s \le 2p$ and $2^{\beta (s-r)} \le 4^{2p}$, therefore
\[ |S(P)| \le 4^{2p} \frac{n^r}{r! \left( c \left\| d \alpha \right\|^{\beta} \right)^{2p-r}} . \]
We have seen that $r \le p$ for any $P$. The number of ordered partitions of $[2p]$ is at most $(2p)^{2p}$, hence summing over all ordered partitions $P$ of $[2p]$ finally shows
\[ \E \left| \sum_{k=m+1}^{m+n} e^{2 \pi i S_k \alpha} \right|^{2p} = \sum_{P} S(P) \le (8p)^{2p} \max_{1 \le r \le p} \frac{n^r}{r! \left( c \left\| d \alpha \right\|^{\beta} \right)^{2p-r}} . \]

Next, we prove Proposition \ref{momentestimate} (ii), i.e.\ we assume \eqref{secondcond}. To estimate \eqref{phiexpansion} under hypothesis \eqref{secondcond} we will need the following lemma.

\begin{lem}\label{fmnslemma} Let $m \ge 0$ and $n, s \ge 1$ be integers, and let $\delta >0$. Consider
\[ f_{m,n,s} (x_1, \dots , x_s) = \sum_{m+1 \le \ell_1< \cdots < \ell_s \le m+n} x_1^{\ell_1} \cdots x_s^{\ell_s} . \]
For a given $x=(x_1, \dots , x_s) \in \mathbb{C}^s$ let
\begin{enumerate}
\item[(i)] $q=q(x)$ denote the maximum number of pairwise disjoint, nonempty intervals of consecutive integers $I_1, \dots , I_q \subseteq [s]$ such that $\left| 1-\prod_{j \in I_r} x_j  \right|<\delta$ for all $1 \le r \le q$,

\item[(ii)] $\displaystyle{K=K(x)=\max \left\{ \prod_{j=a}^s |x_j| \,\, : \,\, 1 \le a \le s \right\} \cup \{ 1 \} }$.
\end{enumerate}
Then
\[ |f_{m,n,s} (x_1, \dots, x_s)| \le K^{m+n+1} \left( \frac{2}{\delta}\right)^s \sum_{r=0}^q \frac{(\delta n)^r}{r!} . \]
\end{lem}

Note that $\delta>0$ is a free parameter, which can be chosen to optimize the estimate. As $\delta \to 0$, each term of the estimate is increasing, however the highest exponent $q$ of $n$ which shows up in the estimate is decreasing.

\begin{proof}[Proof of Lemma \ref{fmnslemma}] We may assume that $x_1, \dots , x_s \neq 0$, otherwise $f_{m,n,s}(x_1, \dots , x_s)=0$. We use induction on $s$. First, let $s=1$, and consider
\[ f_{m,n,1} (x_1)=\sum_{m+1\le \ell_1 \le m+n} x_1^{\ell_1}. \]
If $|1-x_1|<\delta$, then $q=1$. Using the triangle inequality and $|x_1| \le K$ we get
\[ |f_{m,n,1}(x_1)| \le \sum_{m+1\le \ell_1 \le m+n} K^{\ell_1} \le K^{m+n}n \le K^{m+n+1} \frac{2}{\delta} \left( 1 + \delta n \right) , \]
as claimed. If $|1-x_1| \ge \delta$, then $q=0$. In this case we evaluate $f_{m,n,1}(x_1)$ as a partial sum of a geometric series, and obtain
\[ |f_{m,n,1}(x_1)| = \left| \frac{x_1^{m+1}-x_1^{m+n+1}}{1-x_1} \right| \le \frac{K^{m+1}+K^{m+n+1}}{\delta} \le K^{m+n+1} \frac{2}{\delta}, \]
as claimed.

Suppose now that the lemma is true for $s-1$, and let us prove it for $s \ge 2$. Let $x=(x_1, \dots , x_s) \in \mathbb{C}^s$, and consider $q=q(x)$ and $K=K(x)$. We will treat the cases $|1-x_s|<\delta$ and $|1-x_s|\ge \delta$ separately.

Assume first that $|1-x_s|<\delta$. By fixing $\ell_s$ first, and summing over $\ell_1, \dots , \ell_{s-1}$ we get
\[ f_{m,n,s}(x_1, \dots , x_s) = \sum_{m+s \le \ell_s \le m+n} x_s^{\ell_s} \sum_{m+1 \le \ell_1 < \cdots < \ell_{s-1} \le \ell_s-1} x_1^{\ell_1} \cdots x_{s-1}^{\ell_{s-1}} . \]
Note that the inner sum is equal to $f_{m,\ell_s-m-1,s-1} (x_1, \dots , x_{s-1})$. Let $x^*=(x_1, \dots ,$ $x_{s-1}) \in \mathbb{C}^{s-1}$, and consider $q^*=q(x^*)$ and $K^*=K(x^*)$. We have $K^* \le K/|x_s|$ and $q^*=q-1$. Indeed, we can add the singleton $\{ s \}$ to the family of pairwise disjoint, nonempty intervals defining $q^*$. Applying the triangle inequality and the inductive hypothesis we get
\[ \begin{split} | f_{m,n,s}(x_1, x_2, \dots , x_s)| &\le \sum_{m+s \le \ell_s \le m+n} |x_s|^{\ell_s} |f_{m,\ell_s-m-1,s-1} (x_1, x_2, \dots , x_{s-1})| \\ &\le \sum_{m+s \le \ell_s \le m+n} |x_s|^{\ell_s} \left( \frac{K}{|x_s|} \right)^{\ell_s} \left( \frac{2}{\delta} \right)^{s-1} \sum_{r=0}^{q-1} \frac{(\delta (\ell_s-m-1))^r}{r!}. \end{split} \]
Here $|x_s|^{\ell_s} (K/|x_s|)^{\ell_s} \le K^{m+n+1}$, thus
\[ | f_{m,n,s}(x_1, \dots , x_s)| \le K^{m+n+1} \left( \frac{2}{\delta} \right)^{s-1} \sum_{r=0}^{q-1} \frac{\delta^r}{r!} \sum_{m+s\le \ell_s \le m+n} (\ell_s-m-1)^r . \]
The standard estimate
\[ \sum_{m+s\le \ell_s \le m+n} (\ell_s-m-1)^r = \sum_{\ell=s-1}^{n-1} \ell^r \le \frac{n^{r+1}}{r+1} \]
shows
\[ | f_{m,n,s}(x_1, \dots , x_s)| \le K^{m+n+1} \frac{1}{2} \left( \frac{2}{\delta} \right)^{s} \sum_{r=0}^{q-1} \frac{(\delta n)^{r+1}}{(r+1)!} . \]
Reindexing the sum over $r$ finishes the proof of the inductive step in the case $|1-x_s|<\delta$.

Finally, assume $|1-x_s| \ge \delta$. Fixing $m+1 \le \ell_1 < \cdots < \ell_{s-1} \le m+n-1$ first, and summing over $\ell_{s-1}<\ell_s \le m+n$ we obtain
\[ f_{m,n,s}(x_1, \dots , x_s) = \sum_{m+1 \le \ell_1 < \cdots < \ell_{s-1} \le m+n-1} x_1^{\ell_1} \cdots x_{s-1}^{\ell_{s-1}} \frac{x_s^{\ell_{s-1}+1}-x_s^{m+n+1}}{1-x_s} , \]
which yields the recursive formula
\[ \begin{split} f_{m,n,s} (x_1, \dots , x_s) = &\frac{x_s}{1-x_s} f_{m,n-1,s-1} (x_1, \dots , x_{s-1}x_s) \\ &- \frac{x_s^{m+n+1}}{1-x_s} f_{m,n-1,s-1} (x_1, \dots , x_{s-1}) . \end{split} \]
Let $x'=(x_1, \dots , x_{s-1} x_s) \in \mathbb{C}^{s-1}$, and consider $q'=q(x')$ and $K'=K(x')$. It is easy to see that $q' \le q$ and $K' \le K$. Applying the inductive hypothesis and using $\left| x_s/(1-x_s) \right| \le K/\delta$ we get
\begin{equation}\label{fmns1}
\left| \frac{x_s}{1-x_s} f_{m,n-1,s-1} (x_1, \dots , x_{s-1}x_s) \right| \le \frac{K}{\delta} K^{m+n} \left( \frac{2}{\delta} \right)^{s-1} \sum_{r=0}^q \frac{(\delta n)^r}{r!} .
\end{equation}
Let $x''=(x_1, \dots , x_{s-1}) \in \mathbb{C}^{s-1}$, and consider $q''=q(x'')$ and $K''=K(x'')$. It is easy to see that $q'' \le q$ and $K'' \le K/|x_s|$. Applying the inductive hypothesis and using $\left| x_s^{m+n+1}/(1-x_s) \right| \le K |x_s|^{m+n}/\delta$ we get
\begin{align}\label{fmns2}
&\left| \frac{x_s^{m+n+1}}{1-x_s} f_{m,n-1,s-1} (x_1, \dots , x_{s-1}) \right| \\
&  \phantom{99999999999999999} \le \frac{K |x_s|^{m+n}}{\delta} \left( \frac{K}{|x_s|} \right)^{m+n} \left( \frac{2}{\delta} \right)^{s-1} \sum_{r=0}^q \frac{(\delta n)^r}{r!} . \nonumber
\end{align}
Adding \eqref{fmns1} and \eqref{fmns2} we finally get
\[ | f_{m,n,s} (x_1, \dots , x_s)| \le K^{m+n+1} \left( \frac{2}{\delta} \right)^s \sum_{r=0}^q \frac{(\delta n)^r}{r!} . \]
This completes the proof of Lemma \ref{fmnslemma}.
\end{proof}

Let us now return to estimating $S(P)$ in \eqref{phiexpansion} under the hypothesis \eqref{secondcond}. If $\varphi (2 \pi c_j \alpha) =0$ for some $1 \le j \le s$, then $S(P)=0$. Otherwise $S(P)=f_{m,n,s} (x_1, \dots, x_s)$ as in Lemma \ref{fmnslemma} with $x_j = \varphi(2 \pi c_j \alpha)/\varphi (2 \pi c_{j+1} \alpha)$ for $1 \le j \le s-1$, and $x_s = \varphi (2 \pi c_s \alpha)$. First, note that for any $1 \le a \le s$,
\[ \prod_{j=a}^s |x_j| = |\varphi (2 \pi c_a \alpha)| \le 1 , \]
therefore we have $K=K(x)=1$. For an interval of consecutive integers $[a,b] \subseteq [s]$ with $1 \le a \le b<s$ condition \eqref{secondcond} implies
\[ \begin{split} \left| 1-\prod_{j \in [a,b]} x_j \right| = \left| 1- \frac{\varphi(2 \pi c_a \alpha)}{\varphi (2 \pi c_{b+1} \alpha)} \right| &\ge \left| \varphi (2 \pi c_a \alpha) - \varphi (2 \pi c_{b+1} \alpha) \right| \\ &\ge c \left\| d (c_a-c_{b+1}) \alpha \right\| \\&= c \left\| d(\e_a + \e_{a+1} + \cdots + \e_b ) \alpha \right\| . \end{split} \]
Similarly, for an interval of consecutive integers $[a,s] \subseteq [s]$ with $1 \le a \le s$ condition \eqref{secondcond} implies
\[ \begin{split} \left| 1-\prod_{j \in [a,s]} x_j \right| = \left| 1 - \varphi (2 \pi c_a \alpha ) \right| &= \left| \varphi (2 \pi c_a \alpha ) - \varphi (2 \pi 0) \right| \\ &\ge c \left\| d c_a \alpha \right\| \\&= c \left\| d (\e_a + \e_{a+1} + \cdots + \e_s) \alpha \right\| . \end{split} \]
Altogether, for any nonempty interval of consecutive integers $I \subseteq [s]$ we have
\[ \left| 1 - \prod_{j \in I} x_j \right| \ge c \left\| d \left( \sum_{j \in I} \e_j \right) \alpha \right\| = c \left\| d \left( \sum_{i \in \bigcup_{j \in I} P_j} (-1)^{i+1} \right) \alpha \right\| . \]
This estimate gives the idea to choose $\delta=c \left\|d \alpha \right\|$ in Lemma \ref{fmnslemma}. With this choice, $\left|1 - \prod_{j \in I} x_j \right| < \delta$ implies that
\[ \sum_{i \in \bigcup_{j \in I} P_j} (-1)^{i+1} \neq \pm 1, \]
and so $\left| \bigcup_{j \in I} P_j \right| \ge 2$. Hence if $I_1, \dots, I_q \subseteq [s]$ are pairwise disjoint, nonempty intervals of consecutive integers such that $\left| 1 - \prod_{j \in I_r} x_j \right| < \delta$ for every $1 \le r \le q$, then using the fact that $P_1, \dots , P_s$ is a partition of $[2p]$, we get
\[ 2q \le \sum_{r=1}^q \left| \bigcup_{j \in I_r} P_j \right| = \left| \bigcup_{j \in I_1 \cup \cdots \cup I_r} P_j \right| \le 2p. \]
Thus $q=q(x)$ as in Lemma \ref{fmnslemma} satisfies $q \le p$. Applying Lemma \ref{fmnslemma} with $K=1$, $q \le p$ and $\delta = c \left\| d \alpha \right\|$ to \eqref{phiexpansion}, we obtain
\begin{equation}\label{SPestimate}
|S(P)| \le \left( \frac{2}{c \left\| d \alpha \right\|} \right)^s \sum_{r=0}^p \frac{(c \left\| d \alpha \right\| n)^r}{r!}
\end{equation}
for any ordered partition $P=(P_1, \dots , P_s)$ of $[2p]$. Here $s \le 2p$. Since the number of ordered partitions of $[2p]$ is at most $(2p)^{2p}$, summing \eqref{SPestimate} over all ordered partitions $P$ of $[2p]$ finishes the proof of Proposition \ref{momentestimate} (ii):
\[ \E \left| \sum_{k=m+1}^{m+n} e^{2 \pi i S_k \alpha} \right|^{2p} = \sum_{P} S(P) \le (2p)^{2p} \left( \frac{2}{c \left\| d \alpha \right\|} \right)^{2p} \sum_{r=0}^p \frac{(c \left\| d \alpha \right\| n)^r}{r!} . \]
\end{proof}

\subsection{Examples}\label{Examples}

We were able to estimate the moments \eqref{2pmoment} in Proposition \ref{momentestimate} under conditions \eqref{firstcond} and \eqref{secondcond} for the characteristic function $\varphi$ of $X_1$. We now study probability distributions which satisfy those conditions. First of all note that if $X_1$ is integer-valued, then $\varphi (2 \pi x)$ is periodic, e.g.\ 1 is a period. Thus any lower estimate of $1-|\varphi (2 \pi x)|$ and $|\varphi (2 \pi x) - \varphi (2 \pi y)|$ needs to be periodic as well, which explains the use of the distance from the nearest integer function $\left\| \cdot \right\|$. The constant $d>0$ accounts for the fact that the smallest period of $\varphi (2 \pi x)$ or its absolute value might be less than 1.

It is easy to see that \eqref{firstcond} with some $0<\beta<2$ implies $\E X_1^2 = \infty$. Therefore we can only hope to prove \eqref{firstcond} with $0<\beta<2$ for certain ``heavy-tailed'' distributions. On the other hand, \eqref{firstcond} with $\beta=2$ holds in far more general circumstances.

\begin{prop}\label{examplesprop} Let $X_1$ be an integer-valued random variable with characteristic function $\varphi$.
\begin{enumerate}
\item[(i)] If $X_1$ is nondegenerate, then there exist a real number $c>0$ and an integer $d>0$ such that \eqref{firstcond} holds for any $x \in \R$ with $\beta=2$.

\item[(ii)] Let $0<\beta<2$. Suppose there exist constants $K,x_0>0$ such that for any $x \ge x_0$,
\begin{equation}\label{heavytail}
\E \left( X_1^2 I_{\left\{ |X_1| \le x \right\}} \right) \ge K x^{2-\beta} .
\end{equation}
Then there exist a real number $c>0$ and an integer $d>0$ such that \eqref{firstcond} holds for any $x \in \R$.
\end{enumerate}
\end{prop}

\begin{proof} Let $X_2$ be a random variable independent of and with the same distribution as $X_1$. Then
\[ \E e^{2 \pi i x (X_1-X_2)} = \E e^{2 \pi i x X_1} \E e^{-2\pi i x X_2} = |\varphi (2 \pi x)|^2 . \]
By taking the real part of both sides and using a trigonometric identity we obtain
\[ 1 - |\varphi (2 \pi x)|^2 = \E \left( 1 - \cos \left( 2 \pi x (X_1-X_2) \right) \right) = 2 \E \sin^2 \left( \pi x (X_1-X_2) \right) . \]
Let $f: \R \to \R$, $f(x)= \E \sin^2 \left( \pi x (X_1-X_2) \right)$. Since
\[ 1 - |\varphi (2 \pi x)| \ge \frac{1-|\varphi (2 \pi x)|^2}{2} = f(x), \]
it will be enough to find a lower estimate for $f(x)$.

Let $d>0$ denote the greatest common divisor of the (finite or infinite) support of $X_1-X_2$. Note that the nondegeneracy of $X_1$ implies that this support contains a nonzero integer, making $d>0$ well-defined. Clearly, $f$ is periodic with period $1/d$. It is also easy to see that $f(x)=0$ if and only if $x(X_1-X_2) \in \mathbb{Z}$ with probability 1, or equivalently, if and only if $x$ is an integer multiple of $1/d$. Furthermore, $f$ is continuous, which can be seen e.g.\ from Lebesgue's dominated convergence theorem. Hence to prove an estimate of the form
\begin{equation}\label{flowerestimate}
f(x) \ge c \left\| d x \right\|^{\beta}
\end{equation}
for some constant $c>0$ it is enough to prove \eqref{flowerestimate} in an open neighborhood of $0$.

Applying the estimate $\sin^2 (\pi t) \ge 4t^2$, valid for any $|t| \le 1/2$, with $t=x(X_1-X_2)$ gives
\begin{equation}\label{flowerestimate2}
f(x) \ge 4 x^2 \E \left( \left( X_1-X_2 \right)^2 I_{\left\{ |X_1-X_2| \le \frac{1}{2|x|} \right\}} \right) .
\end{equation}

First, we prove (i). We have $\E (X_1-X_2)^2 >0$ (possibly infinite), because $X_1$ is nondegenerate. From the monotone convergence theorem we can see that
\[ \E \left( \left( X_1-X_2 \right)^2 I_{\left\{ |X_1-X_2| \le \frac{1}{2|x|} \right\}} \right) \]
is greater than a fixed positive constant in an open neighborhood of $0$. Therefore \eqref{flowerestimate2} shows that \eqref{flowerestimate} holds with $\beta=2$ and some $c>0$ in an open neighborhood of $0$, and we are done.

Next, we prove (ii). Let $\mu$ denote any median of $|X_1|$, i.e.\ $\P (|X_1|\le \mu) \ge 1/2$ and $\P (|X_1|\ge \mu) \ge 1/2$. If both $2 \mu \le |X_1| \le 1/(2|x|)-\mu$ and $|X_2| \le \mu$, then $|X_1-X_2| \le 1/(2|x|)$ and $(X_1-X_2)^2 \ge X_1^2/4$. Therefore
\[ \left( X_1-X_2 \right)^2 I_{\left\{ |X_1-X_2| \le \frac{1}{2|x|} \right\}} \ge \frac{X_1^2}{4} I_{\left\{ 2 \mu \le |X_1| \le \frac{1}{2|x|}-\mu \right\}} I_{\left\{|X_2| \le \mu \right\}} . \]
Taking the expected value and using the definition of a median we obtain
\[ \begin{split} \E \left( \left( X_1-X_2 \right)^2 I_{\left\{ |X_1-X_2| \le \frac{1}{2|x|} \right\}} \right) &\ge \frac{1}{8} \E \left( X_1^2 I_{\left\{ 2 \mu \le |X_1| \le \frac{1}{2|x|}-\mu \right\}} \right) \\ &\ge \frac{1}{8} \E \left( X_1^2 I_{\left\{|X_1| \le \frac{1}{2|x|}-\mu \right\}} \right) - \frac{\mu^2}{2}. \end{split} \]
Equation \eqref{flowerestimate2} and condition \eqref{heavytail} thus imply that \eqref{flowerestimate} holds with some $c>0$ in an open neighborhood of $0$.
\end{proof}

Next, we study when relation \eqref{secondcond} holds. For the sake of simplicity, assume that $X_1$ is integer-valued, and $\E |X_1|<\infty$. Then $\varphi (2 \pi x)$ has period 1, therefore we may visualize it as a continuously differentiable, closed curve on the Euclidean plane. It is easy to see that the ``self-intersection points'' of this curve, i.e.\ the solutions of the equation $\varphi (2 \pi x) = \varphi (2 \pi y)$, $x \neq y$ will play an important role. Indeed, $|\varphi (2 \pi x) - \varphi (2 \pi y)|$ can be small in two different ways: either $x$ and $y$ are close to each other, or they are close to two different self-intersection points of the curve. In the first case a lower estimate linear in $|x-y|$ can be deduced by assuming $\varphi' \neq 0$ anywhere on $\R$. To handle the second case, we will impose a ``rationality'' and a ``linear independence'' condition on the self-intersection points.

\begin{prop}\label{secondcondprop} Let $X_1$ be an integer-valued random variable with characteristic function $\varphi$ such that $\E |X_1|<\infty$ and $\varphi' \neq 0$ anywhere on $\R$. Let $p>0$ denote the smallest period of $\varphi (2 \pi x)$. Suppose that the equation $\varphi (2 \pi x) = \varphi (2 \pi y)$, $x,y \in [0,p)$, $x \neq y$, has finitely many solutions $(x_1,y_1), \dots, (x_n,y_n)$, and that $x_k-y_k \in \mathbb{Q}$ and $\varphi' (2 \pi x_k)/ \varphi' (2 \pi y_k) \not\in \R$ for any $k=1,\dots, n$. Then there exist a real number $c>0$ and an integer $d>0$ such that \eqref{secondcond} holds for any $x,y \in \R$.
\end{prop}

\begin{proof} Clearly $p>0$ is the reciprocal of the greatest common divisor of the (finite or infinite) support of $X_1$. By considering $pX_1$ instead, we may therefore assume $p=1$. Let $d>0$ be an integer such that $d(x_k-y_k) \in \mathbb{Z}$ for every $k=1,\dots, n$.

The assumption $\E |X_1|<\infty$ implies that $\varphi$ is differentiable, and $\varphi'$ is uniformly continuous. The periodicity of $\varphi$ thus shows that $|\varphi'| \ge K_0$ for some constant $K_0>0$. For any $k=1,\dots, n$ the derivatives $\varphi' (2 \pi x_k)$ and $\varphi' (2 \pi y_k)$ are linearly independent as planar vectors, because $\varphi' (2 \pi x_k)/ \varphi' (2 \pi y_k) \not\in \R$. From the equivalence of finite-dimensional norms, we get that for any $u,v \in \R$,
\begin{equation}\label{derivativeest}
|\varphi'(2 \pi x_k) u - \varphi'(2 \pi y_k) v| \ge K_k \left( |u|+|v| \right)
\end{equation}
with some constant $K_k>0$. Let $K=\min \left\{ K_k \,\, : \,\, 0 \le k \le n \right\}$.

A simple corollary of the uniform continuity of $\varphi'$ is that the convergence
\[ \frac{\varphi (2 \pi t) - \varphi (2 \pi a)}{2 \pi t- 2 \pi a} \to \varphi' (2 \pi a) \]
as $|t-a| \to 0$ is uniform in $t,a \in \R$. In particular, there exists a constant $r>0$ such that whenever $|t-a|<r$, then
\begin{equation}\label{uniformdiff}
\left| \varphi (2 \pi t) - \varphi (2 \pi a) - \varphi'(2 \pi a) (2 \pi t-2 \pi a) \right| \le \pi K |t-a| .
\end{equation}

Consider the compact set
\[ C = \left\{ (x,y) \in [0,1]^2 \,\, : \,\, \varphi (2 \pi x) = \varphi (2 \pi y) \right\} . \]
Note that $C$ consists of the diagonal $x=y$, the points $(0,1)$, $(1,0)$ and the finite point set $(x_k,y_k)$, $k=1, \dots, n$. Let $(x,y) \in [0,1]^2$ be such that $\textrm{dist} \, ((x,y),C)<r/2$, where $\textrm{dist}$ denotes the distance from a point to a set. There are three cases: $(x,y)$ is either close to the diagonal, to $(0,1)$ or $(1,0)$, or to the point $(x_k,y_k)$ for some $k=1,\dots,n$.

First, assume that the distance of $(x,y)$ from the diagonal is less than $r/2$. Then $|x-y|<r$, thus \eqref{uniformdiff} with $t=x$ and $a=y$ implies
\[ \begin{split} \left| \varphi (2 \pi x) - \varphi (2 \pi y) \right| &\ge |\varphi' (2 \pi y)| \cdot |2 \pi x- 2 \pi y| - \pi K |x-y| \\ &\ge \frac{\pi K}{d} |d(x-y)| \ge \frac{\pi K}{d} \left\| d (x-y) \right\| . \end{split} \]
Assume next that the Euclidean distance from $(x,y)$ to $(0,1)$ is less than $r/2$. Then \eqref{uniformdiff} applies with $t=x$ and $a=y-1$. Using the periodicity of $\varphi$ we thus obtain
\[ \begin{split} |\varphi (2 \pi x) - \varphi (2 \pi y)| &= |\varphi (2 \pi x) - \varphi (2 \pi (y-1))| \\ &\ge |\varphi' (2 \pi (y-1))| \cdot |2 \pi x- 2 \pi (y-1)| - \pi K |x-(y-1)| \\ &\ge \frac{\pi K}{d} |d(x-y)+d| \ge \frac{\pi K}{d} \left\| d (x-y) \right\| . \end{split} \]
A similar estimate holds when the distance from $(x,y)$ to $(1,0)$ is less than $r/2$. Finally, assume that the distance from $(x,y)$ to $(x_k,y_k)$ is less than $r/2$ for some $k=1, \dots, n$. In this case \eqref{uniformdiff} applies with $t=x$ and $a=x_k$, and also with $t=y$ and $a=y_k$. Since $\varphi (2 \pi x_k)=\varphi (2 \pi y_k)$, we have
\[ \begin{split} \left| \varphi (2 \pi x) - \varphi (2 \pi y) \right| \ge &\left| \varphi' (2 \pi x_k) (2 \pi x-2 \pi x_k) - \varphi' (2 \pi y_k) (2 \pi y-2 \pi y_k) \right| \\ &- \pi K |x-x_k| - \pi K |y-y_k| . \end{split} \]
Applying \eqref{derivativeest} with $u=x-x_k$ and $v=y-y_k$ we obtain
\[ \begin{split} \left| \varphi (2 \pi x) - \varphi (2 \pi y) \right| &\ge \pi K \left( |x-x_k|+|y-y_k| \right) \\ &\ge \frac{\pi K}{d} | d(x-y)-d(x_k-y_k) | \ge \frac{\pi K}{d} \left\| d (x-y) \right\| . \end{split} \]
Altogether we have shown that for any $(x,y) \in [0,1]^2$ such that $\textrm{dist}\, ((x,y),C)<r/2$ we have
\[ |\varphi (2 \pi x) - \varphi (2 \pi y)| \ge \frac{\pi K}{d} \left\| d (x-y) \right\| . \]
Using the compactness of the corresponding set it is easy to see that for any $(x,y) \in [0,1]^2$ such that $\textrm{dist}\, ((x,y),C) \ge r/2$ we have
\[ |\varphi (2 \pi x) - \varphi (2 \pi y)| \ge c' \left\| d (x-y) \right\| \]
with some constant $c'>0$. Hence \eqref{secondcond} is satisfied with $c=\min \left\{ \pi K/d, c' \right\}$ for any $(x,y) \in [0,1]^2$. By the periodicity of $\varphi$, \eqref{secondcond} is therefore satisfied for all $x,y \in \R$.
\end{proof}

\begin{cor}\label{corollary} Let $X_1$ be a random variable with characteristic function $\varphi$. Suppose that $\P (X_1=a)=\P (X_1=b)=1/2$ for some $a,b \in \mathbb{Z}$ with $a \not\equiv b \pmod{2}$. Then there exist a real number $c>0$ and an integer $d>0$ such that \eqref{secondcond} holds for any $x,y \in \R$.
\end{cor}

\begin{proof} We will show that $X_1$ satisfies the conditions of Proposition \ref{secondcondprop}. The characteristic function of $X_1$ is
\[ \varphi (t) = \frac{1}{2} e^{iat} + \frac{1}{2} e^{ibt} = e^{i \frac{a+b}{2}t} \cos \left( \frac{a-b}{2} t \right) . \]
First, note that
\[ |\varphi' (t)| = \frac{1}{2} \left| a e^{iat} + b e^{ibt} \right| \ge \frac{1}{2} \left| |a|-|b| \right| \ge \frac{1}{2} , \]
therefore $\varphi' \neq 0$ anywhere on $\R$.

Similarly to Proposition \ref{secondcondprop} we may assume that $a$ and $b$ are relatively prime, i.e.\ that the smallest period of $\varphi (2 \pi x)$ is 1. Observe that $a \not\equiv b \pmod{2}$ implies that $a-b$ and $a+b$ are also relatively prime.

Consider the equation $\varphi (2 \pi x) = \varphi (2 \pi y)$, $x \neq y$, which is equivalent to
\begin{equation}\label{phi=phi}
e^{\pi i (a+b)(x-y)} \cos \left( \pi (a-b) x \right) = \cos \left( \pi (a-b) y \right) , \qquad x \neq y .
\end{equation}
We have
\begin{equation}\label{phi'/phi'}
\frac{\varphi' (2 \pi x)}{\varphi' (2 \pi y)} = e^{\pi i (a+b)(x-y)} \frac{i(a+b) \cos (\pi (a-b)x) - (a-b) \sin (\pi (a-b)x)}{i(a+b) \cos (\pi (a-b)y) - (a-b) \sin (\pi (a-b)y)} .
\end{equation}
We distinguish two cases in \eqref{phi=phi}: either $\cos (\pi (a-b)x) = \cos (\pi (a-b)y)=0$, or $\exp (\pi i (a+b)(x-y)) \in \R$. The first case gives finitely many solutions $(x_k,y_k)$ within a period $[0,1)$, each of which satisfies $(a-b)(x_k-y_k) \in \mathbb{Z}$. Since $\sin (\pi (a-b)x_k)$ and $\sin (\pi (a-b)y_k)$ are both $\pm 1$, for these solutions \eqref{phi'/phi'} simplifies to
\[ \frac{\varphi' (2 \pi x_k)}{\varphi' (2 \pi y_k)} = \pm e^{\pi i (a+b) (x_k-y_k)} . \]
By way of contradiction, suppose that this ratio is purely real. Then $(a+b)(x_k-y_k) \in \mathbb{Z}$. Since $a-b$ and $a+b$ are relatively prime, the integrality of $(a-b)(x_k-y_k)$ and $(a+b)(x_k-y_k)$ implies that $x_k-y_k$ is also an integer. This is impossible for $x_k, y_k$ in the period interval $[0,1)$.

Finally, suppose $\exp (\pi i (a+b)(x-y)) \in \R$. It is easy to see that in this case \eqref{phi=phi} also gives finitely many solutions $(x_{\ell},y_{\ell})$ in $[0,1)$, each of which satisfies $(a+b)(x_{\ell}-y_{\ell}) \in \mathbb{Z}$. Since $\exp (\pi i (a+b)(x_{\ell}-y_{\ell}))=\pm 1$, \eqref{phi'/phi'} is purely real if and only if
\begin{multline*}
- \cos (\pi (a-b)x_{\ell}) \sin (\pi (a-b)y_{\ell}) + \cos (\pi (a-b)y_{\ell}) \sin (\pi (a-b)x_{\ell}) \\ = \sin (\pi (a-b)(x_{\ell} - y_{\ell})) =0,
\end{multline*}
which is equivalent to $(a-b)(x_{\ell}-y_{\ell}) \in \mathbb{Z}$. Since $a-b$ and $a+b$ are relatively prime, $(a+b)(x_{\ell} - y_{\ell}) \in \mathbb{Z}$ and $(a-b)(x_{\ell} - y_{\ell}) \in \mathbb{Z}$ would imply $x_{\ell}-y_{\ell} \in \mathbb{Z}$, which is impossible for $x_{\ell},y_{\ell}$ in the period interval $[0,1)$. Therefore the solutions $(x_{\ell}, y_{\ell})$ also satisfy $\varphi' (2 \pi x_{\ell})/\varphi' (2 \pi y_{\ell}) \not\in \R$.
\end{proof}

The simplest case in which the ``rationality'' and the ``linear independence'' conditions on the self-intersection points of $\varphi$ in Proposition \ref{secondcondprop} hold is when $\varphi$ is a simple closed curve, i.e.\ when there are no self-intersection points at all. If $X_1=1$ a.s., then $\varphi (2 \pi x)$ parametrizes the unit circle. Thus if $X_1=1$ has a high enough probability, then $\varphi (2 \pi x)$ will look like a slightly ``deformed'' circle, and we can hope that this slight deformation will not introduce any self-intersection points. It is very easy to turn this idea into a precise proof as follows.

\begin{prop} Let $X_1$ be an integer-valued random variable such that $\E |X_1| < 2 \P (X_1=1)$. Then the characteristic function $\varphi$ of $X_1$ satisfies \eqref{secondcond} with $c=8 \P (X_1=1)- 4\E |X_1|>0$ and $d=1$.
\end{prop}

\begin{proof} We give a direct proof without using Proposition \ref{secondcondprop}. We have
\[ \begin{split} |\varphi (2 \pi x) - \varphi (2 \pi y)| &= \left| \E \left( e^{2 \pi i X_1 x} - e^{2 \pi i X_1 y} \right) \right| \\ &\ge \P (X_1=1) |e^{2 \pi i x} - e^{2 \pi i y}| - \E \left( |e^{2 \pi i X_1 x} - e^{2 \pi i X_1 y}| I_{\{ X_1 \neq 1\}} \right) . \end{split} \]
Using
\[ |e^{2 \pi i X_1 x} - e^{2 \pi i X_1 y}| \le |X_1| \cdot |e^{2 \pi i x} - e^{2 \pi i y}| \]
and
\[ \E \left( |X_1| I_{\{ X_1 \neq 1 \}} \right) = \E |X_1| - \P (X_1=1) \]
we deduce
\[ |\varphi (2 \pi x) - \varphi (2 \pi y)| \ge \left( 2 \P (X_1=1) - \E |X_1| \right) |e^{2 \pi i x} - e^{2 \pi i y}| . \]
Finally, note that
\[ |e^{2 \pi i x} - e^{2 \pi i y}| = 2 |\sin (\pi (x-y))| \ge 4 \| x-y \| . \]
\end{proof}

\section{A Diophantine sum}

To study the discrepancy of the sequence $\{ S_k \alpha \}$, we will combine the Erd\H{o}s--Tur\'an inequality and our estimates for the high moments of an exponential sum in Proposition \ref{momentestimate}. In order to proceed, it will be necessary to estimate sums of the form
\begin{equation}\label{diophsum}
\sum_{h=1}^H \frac{1}{h \left\| h \alpha \right\|^b}
\end{equation}
where $\alpha$ is a given irrational and $0<b \le 1$. Note that in the proof of Theorem \ref{theoremA}, $b$ will be $\beta/2$ in (i), while $b$ will be $1/2$ in (ii). The behavior of the sum \eqref{diophsum} depends on the Diophantine approximation properties of $\alpha$, i.e.\ on how well $\alpha$ can be approximated by rational numbers with small denominators. These properties are encoded in the continued fraction representation of $\alpha$, therefore it is natural to use the theory of continued fractions to estimate \eqref{diophsum}.

Recall that any irrational $\alpha$ has a unique continued fraction representation
\[ \alpha = [a_0;a_1,a_2, \dots ] = a_0 + \cfrac{1}{a_1 + \cfrac{1}{a_2 + \cdots}} \]
where $a_0$ is an integer and $a_i$ is a positive integer for $i \ge 1$. By truncating the infinite continued fraction we obtain the rational numbers
\[ \frac{p_n}{q_n} = [a_0;a_1,a_2, \dots , a_{n-1}] = a_0 + \cfrac{1}{a_1+\cfrac{1}{a_2 + \cfrac{1}{\cdots + \cfrac{1}{a_{n-1}}}}} , \]
for all $n \ge 1$, called the \textit{convergents} to $\alpha$. Their main relevance is that in a certain sense they are the ``best'' rational approximations of $\alpha$.

The fact that $p_n/q_n$ is ``close'' to $\alpha$ implies that $q_n \alpha$ is ``close'' to an integer (namely $p_n$). This gives us the intuition that the largest terms of the sum \eqref{diophsum} are those for which $h=q_n$ for some $n$. Since $1/(h \left\| h \alpha \right\|^b) \ge 1/h$, the best we can hope for is that the contribution of all other terms is at most a constant times $\log H$. We can turn this intuition into a precise statement:

\begin{prop}\label{generaldioph} Let $\alpha = [a_0;a_1,a_2, \dots]$ be the continued fraction representation of an irrational number $\alpha$, and let $p_n/q_n=[a_0;a_1,a_2, \dots , a_{n-1}]$ denote its convergents. For any $0<b \le 1$,
\[ \sum_{0<h<q_n} \frac{1}{h \left\| h \alpha \right\|^b} = O \left( \log^s q_n + \sum_{0<k<n} \frac{1}{q_k \left\| q_k \alpha \right\|^b} \right) , \]
where $s=1$ if $0<b<1$, and $s=2$ if $b=1$. The implied constant depends only on $\alpha$ and $b$.
\end{prop}

To prove Proposition \ref{generaldioph} we need certain facts from the theory of continued fractions. For a proof see e.g.\ \cite{JWC}.

\begin{prop}\label{facts} The convergents $p_n/q_n=[a_0;a_1,a_2, \dots , a_{n-1}]$ of an arbitrary irrational number $\alpha=[a_0;a_1,a_2, \dots]$ satisfy the following:
\begin{enumerate}
\item[(i)] For any $n \ge 2$ we have $\frac{1}{q_{n+1}+q_n} \le \left\| q_n \alpha \right\| = |q_n \alpha - p_n| \le \frac{1}{q_{n+1}}$.
\item[(ii)] For any $n \ge 1$ we have $q_n \alpha - p_n = (-1)^{n+1} |q_n \alpha - p_n|$.
\item[(iii)] The denominators of the convergents satisfy the recurrence $q_{n+1}=a_n q_n + q_{n-1}$ with initial conditions $q_1=1$, $q_2=a_1$.
\item[(iv)] For any $n \ge 2$ we have $p_n q_{n-1}-q_n p_{n-1} = (-1)^n$. In particular, $p_n$ and $q_n$ are relatively prime.
\end{enumerate}
\end{prop}
\qed

\begin{proof}[Proof of Proposition \ref{generaldioph}] Let $k \ge 3$, and consider the sum
\begin{equation}\label{qkqk+1}
\sum_{q_k \le h < q_{k+1}} \frac{1}{h \left\| h \alpha \right\|^b}.
\end{equation}
Let $\e_k = q_k \alpha - p_k$. Note that $\left\| h \alpha \right\| = \left\| hp_k/q_k + h \e_k/q_k \right\|$. Here $hp_k/q_k$ is an integer multiple of $1/q_k$, and $\left| h \e_k/q_k \right| < q_{k+1}|\e_k|/q_k \le 1/q_k$ for any $q_k \le h < q_{k+1}$. The assumption $k \ge 3$ ensures $q_k \ge 2$. Hence $\left\| h \alpha \right\|$ is basically determined by the residue class of $hp_k$ modulo $q_k$. Since $\textrm{sign } \e_k = (-1)^{k+1}$, the residue classes $0$ and $(-1)^k$ will require special treatment. It is thus natural to decompose the sum \eqref{qkqk+1} using the index sets
\[ \begin{split} A &= \left\{ q_k \le h < q_{k+1} \,\, : \,\, h p_k \equiv 0 \pmod{q_k} \right\} ,\\
B &= \left\{ q_k \le h < q_{k+1} \,\, : \,\, h p_k \equiv (-1)^k \pmod{q_k} \right\} ,\\
C &= \left\{ q_k \le h < q_{k+1} \,\, : \,\, h p_k \not\equiv 0, (-1)^k \pmod{q_k} \right\} . \end{split} \]

First, consider the sum over $h \in A$. Since $p_k$ and $q_k$ are relatively prime, $A$ only contains integral multiples of $q_k$. For any $h=aq_k \in A$, $a \ge 1$ we thus have
\[ \left\| h \alpha \right\| = \left\| \frac{0}{q_k} + \frac{a q_k \e_k}{q_k} \right\| = a |\e_k| = a \left\| q_k \alpha \right\| , \]
and therefore
\begin{equation}\label{sumoverA}
\sum_{h \in A} \frac{1}{h \left\| h \alpha \right\|^b} \le \sum_{a=1}^{\infty} \frac{1}{a q_k \left( a \left\| q_k \alpha \right\| \right)^b} = O \left( \frac{1}{q_k \left\| q_k \alpha \right\|^b} \right) .
\end{equation}

Next, let us estimate the sum over $h \in B$. By taking the equation $p_k q_{k-1} - q_k p_{k-1}=(-1)^k$ from Proposition \ref{facts} (iv) modulo $q_k$, we find that the multiplicative inverse of $p_k$ modulo $q_k$ is $(-1)^kq_{k-1}$, hence every element of $B$ is congruent to $q_{k-1}$ modulo $q_k$. In fact $B=\{ aq_k + q_{k-1} \, : \, 1 \le a \le a_{k}-1 \}$, since $a_k q_k+q_{k-1}=q_{k+1}$ is outside the interval $q_k \le h<q_{k+1}$. From Proposition \ref{facts} (i, iii) we deduce that $a_k q_k |\e_k| \le 1-q_{k-1}|\e_k|$. For any $h=aq_k+q_{k-1} \in B$ we thus have
\[ \left\| h \alpha \right\| = \left\| \frac{(-1)^k}{q_k} + \frac{h \e_k}{q_k} \right\| = \frac{1-(aq_k+q_{k-1})|\e_k|}{q_k} \ge (a_k-a)|\e_k| . \]
Therefore
\begin{equation}\label{sumoverB}
\sum_{h \in B} \frac{1}{h \left\| h \alpha \right\|^b} \le \sum_{a=1}^{a_k-1} \frac{1}{a q_{k} \left( (a_k-a) \left\| q_k \alpha \right\| \right)^b} = O \left( \frac{1}{q_k \left\| q_k \alpha \right\|^b} \right) .
\end{equation}

Finally, we need to estimate the sum over $h \in C$. The congruence conditions in the definition of $C$ imply that for any $h \in C$,
\[ \left\| h \alpha \right\| = \left\| \frac{hp_k}{q_k} + \frac{h \e_k}{q_k} \right\| \ge \frac{1}{2} \left\| \frac{hp_k}{q_k} \right\| . \]
For any integer $a \ge 1$ we therefore have
\begin{equation}\label{aqka+1qk}
\sum_{\substack{aq_k \le h < (a+1)q_k \\ h \in C}} \frac{1}{h \left\| h \alpha \right\|^b} \le \sum_{aq_k < h < (a+1)q_k} \frac{2^b}{aq_k \left\| hp_k/q_k \right\|^b} .
\end{equation}
Since $p_k$ and $q_k$ are relatively prime, as $h$ runs in the interval $aq_k < h < (a+1)q_k$, the numbers $h p_k$ attain each nonzero residue class modulo $q_k$ exactly once. Considering the cases $0<b<1$ and $b=1$ separately, we find that the right hand side of \eqref{aqka+1qk} can hence be estimated as
\[  \frac{2^b}{a q_k} \sum_{j=1}^{q_k-1} \frac{1}{\left\| j/q_k \right\|^b}  \le \frac{2 \cdot 2^b}{a q_k} \sum_{1 \le j \le q_k/2} \frac{1}{\left( j/q_k \right)^b} = O \left( \frac{\log^{s-1} q_k}{a} \right) . \]
Summing over $1 \le a \le a_k$ we obtain
\begin{equation}\label{sumoverC}
\sum_{h \in C} \frac{1}{h \left\| h \alpha \right\|^b} = O \left( \log^{s-1} q_k \log a_k \right) .
\end{equation}

Adding \eqref{sumoverA}, \eqref{sumoverB} and \eqref{sumoverC} we get
\[ \sum_{q_k \le h < q_{k+1}} \frac{1}{h \left\| h \alpha \right\|^b} = O \left( \log^{s-1} q_k \log a_k + \frac{1}{q_k \left\| q_k \alpha \right\|^b} \right) . \]
Summing over $3 \le k \le n-1$ we obtain
\begin{align}\label{0hqn}
&\sum_{0<h<q_n} \frac{1}{h \left\| h \alpha \right\|^b} \\
&\phantom{9999} = \sum_{0<h<q_3} \frac{1}{h \left\| h \alpha \right\|^b} + O \left( \sum_{k=3}^{n-1} \left( \log^{s-1} q_k \log a_k + \frac{1}{q_k \left\| q_k \alpha \right\|^b} \right) \right) \nonumber.
\end{align}
Here the sum over $0<h<q_3$ is $O(1)$, because $q_3$ is a constant depending only on $\alpha$. The recurrence in Proposition \ref{facts} (iii) shows $q_{n} \ge a_{n-1} q_{n-1}$, and iterating this inequality we get
\[ q_n \ge a_{n-1} a_{n-2} \cdots a_3 q_3 . \]
Hence $\sum_{k=3}^{n-1} \log^{s-1} q_k \log a_k = O (\log^s q_n)$, and so \eqref{0hqn} simplifies to
\[ \sum_{0<h<q_n} \frac{1}{h \left\| h \alpha \right\|^b} = O \left( \log^s q_n + \sum_{0<k<n} \frac{1}{q_k \left\| q_k \alpha \right\|^b} \right) . \]
\end{proof}

\begin{cor}\label{gammadioph} Let $\alpha$ be irrational and $0<b \le 1$. Suppose there exist constants $\gamma \ge 1$ and $C>0$ such that $\left\| q \alpha \right\| \ge C q^{- \gamma}$ for every $q \in \mathbb{N}$. Then
\[ \sum_{h=1}^H \frac{1}{h \left\| h \alpha \right\|^b} = \left\{ \begin{array}{ll} O \left( \log^s H \right) & \textrm{if } \gamma \le 1/b, \\ O \left( H^{b \gamma -1} \right) & \textrm{if } \gamma > 1/b, \end{array} \right. \]
where $s=1$ if $0<b<1$, and $s=2$ if $b=1$. The implied constants depend only on $\alpha$, $b$ and $\gamma$.
\end{cor}

\begin{proof} Let $p_n/q_n$ denote the convergents to $\alpha$. Consider the two consecutive convergent denominators such that $q_{n-1} \le H < q_n$. Proposition \ref{generaldioph} implies
\begin{equation}\label{generaldiophest}
\sum_{h=1}^H \frac{1}{h \left\| h \alpha \right\|^b} = O \left( \log^s q_n + \sum_{0<k<n} \frac{1}{q_k \left\| q_k \alpha \right\|^b} \right) .
\end{equation}
Proposition \ref{facts} (i) shows that $C q_{n-1}^{- \gamma} \le \left\| q_{n-1} \alpha \right\| \le 1/q_n$. Rearranging we get $q_n \le q_{n-1}^{\gamma}/C \le H^{\gamma}/C$. Therefore the first error term in \eqref{generaldiophest} satisfies $\log^s q_n = O \left( \log^s H \right)$.

In the second error term in \eqref{generaldiophest} we have
\[ \frac{1}{q_k \left\| q_k \alpha \right\|^b} \le C^{-b} q_k^{b \gamma -1} = O \left( q_k^{b \gamma -1} \right) . \]
If $\gamma \le 1/b$, then
\[ \sum_{0<k<n} \frac{1}{q_k \left\| q_k \alpha \right\|^b} = O \left( \sum_{0<k<n} q_k^{b \gamma -1} \right) = O \left( n \right) . \]
The recurrence in Proposition \ref{facts} (iii) shows that $q_n$ is at least as large as the $n$th Fibonacci number, therefore $n=O (\log q_{n-1}) = O \left( \log H \right)$. Hence \eqref{generaldiophest} simplifies to
\[ \sum_{h=1}^H \frac{1}{h \left\| h \alpha \right\|^b} = O \left( \log^s H + \log H \right) = O \left( \log^s H \right) . \]

Finally, assume $\gamma >1/b$. Proposition \ref{facts} (iii) shows that $q_{k+2} \ge q_{k+1}+q_k \ge 2 q_k$. In particular, any interval of the form $\left[ 2^{\ell},2^{\ell+1}\right)$ contains at most two convergent denominators. Hence
\[ \begin{split} \sum_{0<k<n} \frac{1}{q_k \left\| q_k \alpha \right\|^b} &= O \left( \sum_{0<k<n} q_k^{b \gamma -1} \right) = O \left( \sum_{\substack{\ell \\ 2^{\ell} \le q_{n-1}}} \sum_{2^{\ell} \le q_k < 2^{\ell +1}} q_k^{b \gamma -1} \right) \\ &= O \left( \sum_{\substack{\ell \\ 2^{\ell} \le H}} 2^{(\ell +1) (b \gamma -1)} \right) = O \left( H^{b \gamma -1} \right) . \end{split} \]
Thus in this case \eqref{generaldiophest} gives
\[ \sum_{h=1}^H \frac{1}{h \left\| h \alpha \right\|^b} = O \left( \log^s H + H^{b \gamma -1} \right) = O \left( H^{b \gamma -1} \right) . \]
\end{proof}

\section{Proof of the upper bounds}

In what follows, $K$ will denote positive constants, not always the same, depending (at most) on $\alpha$ and the distribution of $X_1$.
We first show
\begin{lem}\label{lem:1} Let $X_1, X_2, \ldots$ and $\alpha$ be as in Theorem \ref{theoremA} and assume (\ref{firstcond}).
Then for any integers $\ell \ge  0$ and $p\ge 1$ we have
\begin{multline}\label{5}
\left\| \max_{2^{\ell}\le N\le 2^{\ell +1}} ND_N(\{S_k\alpha\})\right\|_{2p} \\ \le K \left( 2^{\ell (1-\delta)}p^\delta +2^{\ell /2} \sqrt{p} \sum_{h=1}^{[K 2^{(\ell +1) \delta}/p^\delta]} \frac{1}{h\|dh\alpha\|^{\beta/2}} \right)
\end{multline}
where $\delta=1/(\beta\gamma)$. If instead of (\ref{firstcond}) we assume (\ref{secondcond}), then (\ref{5}) holds with $\beta=1$.
\end{lem}

\begin{proof} Assume first (\ref{firstcond}). Then
by Proposition \ref{momentestimate} (i), for any integers $m\ge 0$ and $n,h,p\ge 1$ we have
\begin{equation}\label{1}
\mathbb E \left|\sum_{k=m+1}^{m+n} e^{2\pi i S_k h\alpha}\right|^{2p} \le (8p)^{2p} \max_{1\le r\le p} \frac{n^r}{r! \left( c\|d h \alpha\|^\beta\right)^{2p-r}}.
\end{equation}
Let
\begin{equation}\label{H}
H_n= C^{1/\gamma}d^{-1}(cn/p)^{1/(\beta\gamma)}.
\end{equation}
We claim that for any $1\le h\le H_n$ and $0\le r<p$ we have
\begin{equation}\label{incr}
\frac {n^r}{r! \left(c\|dh\alpha\|^\beta\right)^{2p-r}}\le \frac {n^{r+1}}{(r+1)! \left(c\|dh\alpha\|^\beta\right)^{2p-r-1}}.
\end{equation}
To see this, we note that (\ref{incr}) is equivalent to $r+1\le n c \|dh\alpha\|^\beta$ and for $1\le h \le H_n$ and $0\le r<p$ we know by (\ref{H}) and the assumptions of Theorem \ref{theoremA} that
$$ \|dh\alpha\|^\beta \ge C^\beta (dh)^{-\beta\gamma} \ge C^\beta (dH_n)^{-\beta\gamma} =p/(cn) \ge (r+1)/(cn).$$
Thus the maximum in (\ref{1}) is reached for $r=p$ and consequently
\begin{equation}\label{mom1}
\left\| \sum_{k=m+1}^{m+n} e^{2\pi i S_k h\alpha}\right\|_{2p} \le K \sqrt{np} \frac{1}{\|dh\alpha\|^{\beta/2}}
\end{equation}
for all $m\ge 0$, $n,p \ge 1$ and $1\le h\le H_n$.  Set now
$$D_N(\alpha)=D_N(\{S_k\alpha\}), \quad T_h(N, \alpha)=\sum_{k=1}^N e^{2\pi i h S_k \alpha}.$$
By the Erd\H{o}s--Tur\'an inequality we
have
$$
ND_N(\alpha)\le 6\left(\frac{N}{[H_N]} + \sum_{h=1}^{[H_N]} \frac{1}{h} |T_h(N, \alpha)|\right)
$$
and consequently
\begin{equation}\label{6}
\max_{2^{\ell}\le N\le 2^{\ell +1}}ND_N(\alpha) \le K\left( 2^{\ell (1-\delta)}p^\delta + \sum_{h=1}^{[K 2^{(\ell +1)\delta}/p^\delta]} \frac{1}{h}\max_{2^{\ell} \le N\le 2^{\ell +1}} |T_h(N, \alpha)| \right).
\end{equation}
(Note that $N/[H_N]\sim Kp^\delta N^{1-\delta}$ and thus its maximum for $2^{\ell}\le N\le 2^{\ell+1}$ is $\le K 2^{\ell(1-\delta)}p^\delta$.)
By (\ref{mom1})
\begin{equation}\label{7}
\|T_h(N, \alpha)\|_{2p} \le   K \sqrt{Np} \frac{1}{ \|d h\alpha\|^{\beta/2}}.
\end{equation}
Since this remains valid for shifted sums $T_h(N, M, \alpha)=\sum_{k=M+1}^{M+N} e^{2\pi ihS_k \alpha}$ as well, the
Erd\H{o}s--Stechkin inequality \cite{MO} yields
\begin{equation*}
\left\|\max_{2^{\ell}\le N\le 2^{\ell +1}}T_h(N, \alpha)\right\|_{2p} \le  K 2^{\ell /2} \sqrt{p} \frac{1}{\|dh\alpha\|^{\beta/2}} .
\end{equation*}
Substituting this in (\ref{6}) it follows that
$$\left\| \max_{2^{\ell}\le N\le 2^{\ell +1}} ND_N(\alpha)\right\|_{2p} \le K \left(2^{\ell (1-\delta)}p^\delta +2^{\ell/2} \sqrt{p} \sum_{h=1}^{[K 2^{(\ell+1)\delta}/p^\delta]} \frac{1}{h\|dh\alpha\|^{\beta/2}}\right),$$
and thus (\ref{5}) is proved under condition (\ref{firstcond}) in  Theorem \ref{theoremA}.

If instead of (\ref{firstcond}) we assume (\ref{secondcond}), the proof of (\ref{5}) is essentially the same. In this case in Proposition \ref{momentestimate} we have (\ref{mom(ii)}) instead of (\ref{mom(i)}), which implies, in view of the monotonicity relation (\ref{incr}), that (\ref{mom1}) remains valid with $\beta=1$ and a different constant $K$. The rest of the proof of (\ref{5}) requires no change.
\end{proof}

\begin{proof}[Proof of Theorem \ref{theoremA}]
Assume first that (\ref{firstcond}) holds. We will deal with the cases $\gamma>2/\beta$ and $1 \le \gamma\le 2/\beta$ separately.

Assume $\gamma>2/\beta$. Then $\delta<1/2$ and by Lemma \ref{lem:1} and Corollary \ref{gammadioph},
\[ \begin{split} \left\| \max_{2^{\ell}\le N\le 2^{\ell +1}} ND_N(\{S_k\alpha\})\right\|_{2p} &\le K \left(2^{\ell (1-\delta)}p^\delta + 2^{\ell /2} \sqrt{p} \left(2^{(\ell+1)\delta}/p^\delta\right)^{\beta\gamma/2-1}\right)\\
&=K \left(2^{\ell (1-\delta)}p^\delta + 2^{\ell /2} \sqrt{p}\, 2^{(\ell +1)(1/2-\delta)}p^{\delta-1/2}\right)\\
&\le K 2^{\ell (1-\delta)} p^\delta \end{split} \]
for any integers $\ell \ge 0$ and $p \ge 1$. Choosing $p\sim \log \ell$ and using the Markov inequality we get, for a sufficiently large constant $B>0$,
\[ \P \left( \max_{2^{\ell}\le N\le 2^{\ell+1}} ND_N(\{S_k\alpha\})\ge B 2^{\ell (1-\delta)} p^\delta \right)\le \left( \frac{K2^{\ell (1-\delta)} p^\delta}{B2^{\ell (1-\delta)} p^\delta}  \right)^{2p} \le 4^{-2p}\le \ell^{-2}.\]
Using the Borel--Cantelli lemma we get
$$  \max_{2^{\ell}\le N\le 2^{\ell +1}} ND_N(\{S_k\alpha\}) =O\left( 2^{\ell (1-\delta)} p^\delta \right)=O\left( 2^{\ell (1-\delta)} (\log\log 2^{\ell})^\delta \right) \qquad \text{a.s.},$$
proving the second estimate in (\ref{main}).

Assume now $1\le \gamma\le 2/\beta$. Then $\delta\ge 1/2$ and thus using Lemma \ref{lem:1} and Corollary \ref{gammadioph}  we get
\[ \left\| \max_{2^{\ell}\le N\le 2^{\ell +1}} ND_N(\{S_k\alpha\})\right\|_{2p} \le K \left(2^{\ell /2}p^\delta+2^{\ell /2} \sqrt{p}\, \ell^s \right) \le K 2^{\ell /2} \ell^s \sqrt{p} \]
for any integers $\ell \ge 1$, $1 \le p \le \ell^{s/\delta}$, where $s=1$ if $0<\beta<2$, and $s=2$ if $\beta =2$. Choosing again $p\sim \log \ell$ and using the Markov inequality we get, for a sufficiently large constant $B$,
\begin{equation*}
 \P \left( \max_{2^{\ell} \le N\le 2^{\ell +1}} ND_N( \{S_k\alpha\}) \ge B 2^{\ell /2} \ell^s \sqrt{p} \right) \le 4^{-2p}\le \ell^{-2}.
\end{equation*}
Hence the Borel--Cantelli lemma yields the first estimate in (\ref{main}).

If in Theorem \ref{theoremA} we assume (\ref{secondcond}), the argument is the same, using the fact that in this case by Lemma \ref{lem:1} we have (\ref{5}) with $\beta=1$.
\end{proof}

Corollary \ref{corollary} shows that the random variable with distribution $\P (X_1=1)= \P (X_1=2)=1/2$ satisfies the conditions of Theorem \ref{theoremA} (ii), proving the upper bounds in Proposition \ref{mainprop1}. To see the upper bounds in Proposition \ref{mainprop2} note that the condition $c_1 x^{-\beta} \le \P (|X_1| \ge x)$ clearly implies \eqref{heavytail}, and so according to Proposition \ref{examplesprop} (ii), Theorem \ref{theoremA} (i) applies. Finally, the upper bounds in Proposition \ref{mainprop3} follow from Theorem \ref{theoremA} (i) with $\beta =2$ and Proposition \ref{examplesprop} (i).

\section{Proof of the lower bounds}

We start by proving two general lower bounds of independent interest.

\begin{lem}\label{proplower2} Let $X_1, X_2, \ldots$ be integer-valued random variables, let $S_k=\sum_{j=1}^k X_j$, and let $\alpha \in \R$ be irrational such that $\left\| q \alpha \right\| \le C q^{-\gamma}$ for infinitely many $q \in \mathbb{N}$ with some constants $\gamma \ge 1$ and $C>0$. Assume that $S_k=O \left( \psi (k) \right)$ a.s., where $\psi (k)$ is a nondecreasing sequence of positive reals. Then
\begin{equation*}
D_N \left( \{ S_k \alpha \} \right) = \Omega \left( \psi (N+1)^{-1/\gamma} \right) \quad \text{a.s.}
\end{equation*}
\end{lem}

Note that here we allow $X_1, X_2, \ldots$ to be degenerate, in which case the sequence $(S_n)$ is a deterministic sequence of integers.

\begin{proof}[Proof of Lemma \ref{proplower2}]
If $\psi (k)=O(1)$, then the sequence $\{ S_k \alpha \}$ attains only finitely many points, and thus trivially $D_N (\{ S_k \alpha \}) = \Omega (1)$ a.s. We may therefore assume $\psi (k) \to \infty$ as $k \to \infty$. Let $K>0$ be a random variable such that $|S_k| \le K \psi (k)$ for every $k \in \mathbb{N}$.

Let $q \in \mathbb{N}$ with $q > (3CK \psi (1))^{1/\gamma}$ be such that $\left\| q \alpha \right\| = |q \alpha -p| \le C q^{-\gamma}$, where $p=p(q)$ denotes the integer closest to $q \alpha$. Let $N=N(q)$ be the largest positive integer such that $\psi (N) < q^{\gamma}/(3CK)$, i.e.\ $\psi (N) < q^{\gamma}/(3CK) \le \psi (N+1)$. Note that
\[ \left| S_k \alpha - \frac{S_k p}{q} \right| = |S_k| \frac{\left\| q \alpha \right\|}{q} \le K \psi (N) \frac{C q^{-\gamma}}{q} < \frac{1}{3q} \]
holds for any $k=1,\dots, N$. This means that $S_k \alpha$ is in the open neighborhood of some integral multiple of $1/q$ with radius $1/(3q)$. In particular, none of the points $\{ S_k \alpha \}$, $k=1, \dots, N$ lies in $\left[ 1/(3q), 2/(3q) \right] \subset [0,1]$. By the definition of discrepancy we thus have
\begin{equation}\label{DNlowerestimate}
D_N (\left\{ S_k \alpha \right\}) \ge \frac{1}{3q} \ge \frac{1}{3 \left( 3 CK \psi (N+1) \right)^{1/\gamma}} .
\end{equation}
Clearly there are only finitely many $q \in \mathbb{N}$ for which $N(q)$ is a given integer, therefore the existence of infinitely many $q \in \mathbb{N}$ with $\left\| q \alpha \right\| \le C q^{-\gamma}$ implies the existence of infinitely many $N \in \mathbb{N}$ for which \eqref{DNlowerestimate} holds.
\end{proof}

\begin{lem} \label{proplower3} Let $X_1, X_2, \ldots$ be integer-valued i.i.d.\ random variables with characteristic function $\varphi$, and let $S_k=\sum_{j=1}^k X_j$. Suppose that $|1-\varphi (x)| \le c |x|^{\beta}$ for all $x \in \R$ with some constants $c>0$ and $0<\beta \le 2$. Assume further that $\left\| q \alpha \right\| \le Cq^{-\gamma}$ for infinitely many $q \in \mathbb N$  with some constants $\gamma \ge 1$ and $C>0$. Then
\begin{equation*}
D_N \left( \{ S_k \alpha \} \right) = \Omega \left( N^{-1/(\beta\gamma)} \right) \quad \text{a.s.}
\end{equation*}
\end{lem}

\begin{proof} By the assumption on $\varphi$, the characteristic function of $S_n/n^{1/\beta}$ satisfies $|1-\varphi^n (x/n^{1/\beta})| \le n |1-\varphi (x/n^{1/\beta})| \le c |x|^{\beta}$ for any $x \in \R$. Using a well-known method to estimate the tail probabilities of a random variable in terms of its characteristic function (see e.g.\ \cite[p.\ 171--172, Proposition 8.29]{BRE}) we obtain
\begin{equation}\label{sntails}
\P \left( |S_n|/n^{1/\beta} >t \right) \ll t \int_0^{1/t} \left( 1-\mathrm{Re} \varphi^n (x/n^{1/\beta}) \right) \, \mathrm{d}x \ll c t^{-\beta}
\end{equation}
for any $t>0$ with a universal implied constant.

Let $M_n=\max_{1 \le k \le n} |S_k|$, and let $\mu_k$ denote a median of $S_k$. From \eqref{sntails} we get $|\mu_k| \le c' k^{1/\beta}$ with some constant $c'>0$. L\'evy's inequality (see e.g.\ \cite[p.\ 259]{LO}) and \eqref{sntails} hence give, for all $t>c'$,
\[ \begin{split} \P \left( M_n \ge t n^{1/\beta} \right) &\le \P \left( \max_{1 \le k \le n}|S_k + \mu_{n-k}| \ge (t-c') n^{1/\beta}  \right) \\ &\le 2 \P \left( |S_n| \ge (t-c')n^{1/\beta} \right) \\ &\ll c(t-c')^{-\beta} . \end{split} \]
In particular, there exist constants $C_1, C_2>0$ such that $\P \left( M_n \ge C_1 n^{1/\beta} \right) \le 1-C_2$ for all $n \in \mathbb{N}$.

We will use a trivial version of the Borel--Cantelli lemma stating that if $A_1, A_2, \ldots$ are arbitrary events with $\P(A_k)\ge \lambda$ $(k=1, 2, \ldots)$, then with probability $\ge \lambda$, infinitely many $A_k$ will occur. By the assumptions, there exists an infinite subset $H$ of $\mathbb N$ such that $\left\| q \alpha \right\| \le C q^{-\gamma}$ for $q \in H$. For each $q\in H$, let $N=N(q)= [aq^{\beta\gamma}]$, where $a$ is a small constant. Thus letting $A_q= \left\{M_{N(q)}< C_1 N(q)^{1/\beta}\right\}$, we have $\P(A_q) \ge C_2$ for all $q\in H$, and thus with probability $\ge C_2$ infinitely many of the $A_q$, $q\in H$ occur. By the Hewitt--Savage zero-one law (see e.g.\ \cite[p.\ 64, Corollary 3.50]{BRE}), this is actually true with probability 1. Choose now such a $q$; then $\left\| q \alpha \right\| = |q \alpha -p| \le C q^{-\gamma}$, where $p=p(q)$ denotes the integer closest to $q \alpha$. Hence for $N=N(q)$ on the set $A_q$ we have, for any $1\le k\le N$,
\begin{equation}\label{diff}
\left |S_k\alpha-\frac{S_k p}{q}\right|\le C \frac{|S_k|}{q^{\gamma+1}}\le C \frac{|M_N|}{q^{\gamma+1}} \le \frac{C C_1 N^{1/\beta}}{q^{\gamma+1}} \le
\frac{C C_1a^{1/\beta}}{q} \le \frac{1}{3q}
\end{equation}
provided $a$ is small enough. Since the $X_i$ are  integer-valued, the points $S_k p/q$ are integer multiples of $1/q$  and thus by (\ref{diff}) the points $S_k \alpha$ $(1\le k\le N)$ differ from each other by $\ge 1/(3q)$, and consequently with probability 1,
\begin{equation*}
D_N (\{S_k \alpha\}) \ge \frac{1}{3q}\ge C_3 N^{-1/(\beta\gamma)}
\end{equation*}
with some constant $C_3>0$ for infinitely many $N$, as stated.
\end{proof}

The lower bounds in Propositions \ref{mainprop1} (i),  \ref{mainprop2} (i), \ref{mainprop3} (i) are all special cases of Proposition \ref{generallower}. The lower bound in Proposition \ref{mainprop1} (ii) follows from Lemma \ref{proplower2} with $\psi (k)=k$. The lower bound in Proposition \ref{mainprop3} (ii) is a corollary of Lemma \ref{proplower3} with $\beta =2$. Indeed, note that if $\E X_1=0$ and $\E X_1^2 < \infty$, then $\varphi (x)=1-\E X_1^2 x^2 (1+o(1))$ as $x \to 0$, and hence $|1-\varphi (x)| \le cx^2$ with some constant $c>0$.

Finally, we claim that under the conditions of Proposition \ref{mainprop2} we have $|1-\varphi (x)| \le c|x|^{\beta}$ with some $c>0$. The lower bound in Proposition \ref{mainprop2} (ii) will thus follow from Lemma \ref{proplower3}. To see this, consider
\begin{equation}\label{finaleq}
|1-\varphi (x)| = \sqrt{\left( \E (1-\cos (xX_1)) \right)^2 + \left( \E \sin (xX_1) \right)^2} .
\end{equation}
To estimate the first term in \eqref{finaleq}, we will use $1-\cos (xX_1) \le (xX_1)^2/2$ if $|xX_1|<1$, and $1-\cos (xX_1) \le 2$ otherwise. Hence
\[ \begin{split} 1-\cos (xX_1) &\le \frac{x^2}{2} X_1^2 I_{\{ |X_1|<1/|x| \}} + 2 I_{\{ |X_1| \ge 1/|x| \}}, \\ \E (1-\cos (xX_1)) &\le \frac{x^2}{2} \E \left( X_1^2 I_{\{ |X_1|<1/|x| \}} \right) +2\P (|X_1| \ge 1/|x|) \\ &\le x^2 \int_0^{1/|x|} t \P (|X_1| \ge t) \, \mathrm{d}t +2\P (|X_1| \ge 1/|x|) .  \end{split} \]
The assumption $\P (|X_1| \ge x) \le c_2 x^{-\beta}$ thus shows that $\E (1-\cos (xX_1)) \ll |x|^{\beta}$. To estimate the second term in \eqref{finaleq}, we will use $\sin (xX_1) = xX_1 +O(|xX_1|^3)$ if $|xX_1| < 1$, and $|\sin (xX_1)| \le 1$ otherwise. We thus obtain
\[ \begin{split} \left| \E \sin (xX_1)\right| &\ll x \E (X_1 I_{\{ |X_1| < 1/|x| \}}) + |x|^3 \E \left( |X_1|^3 I_{\{ |X_1|<1/|x| \}} \right) + \P \left( |X_1| \ge 1/|x| \right) \\ &\le x \E (X_1 I_{\{ |X_1| < 1/|x| \}}) + |x|^3 \int_0^{1/|x|} 3t^2 \P (|X_1| \ge t) \, \mathrm{d}t + \P \left( |X_1| \ge 1/|x| \right) . \end{split} \]
By the assumption $\P (|X_1| \ge x) \le c_2 x^{-\beta}$ the last two terms are indeed $\ll |x|^{\beta}$. Considering the cases $0<\beta<1$, $\beta =1$ and $1<\beta<2$ separately, it is not difficult to see that the first term is also $\ll |x|^{\beta}$. Hence $|1-\varphi (x)| \ll |x|^{\beta}$, as claimed.

\end{document}